\documentclass[12pt]{amsart}
\usepackage{amssymb,amsmath,amsthm}  
\usepackage{graphicx}
\usepackage{bm}
\usepackage{enumerate}
\usepackage{braket}
\usepackage{amsaddr}
\usepackage{amscd}
\usepackage[dvipsnames]{xcolor}
\usepackage[mathscr]{eucal}
\usepackage{epstopdf}
\usepackage[hidelinks]{hyperref}
\usepackage{subcaption}
\usepackage{cleveref}
\hypersetup{
  colorlinks   = true, 
  urlcolor     = blue, 
  linkcolor    = blue, 
  citecolor   = PineGreen 
}
\numberwithin{equation}{section}
\newtheorem{theorem}{Theorem}[section] 

\newtheorem{lemma}[theorem]{Lemma}

\newtheorem{remark}[theorem]{Remark}
\newtheorem{assumption}[theorem]{Assumption}

\newcommand\norm[1]{\left\lVert#1\right\rVert}
\newcommand{\tu}{\textup}

\newcommand{\dd}{\displaystyle}
\newcommand{\bfa}[1]{\boldsymbol{#1}} 			%

\newcommand{\e}{\epsilon}
\newcommand{\ddiv}{\text{div}}     				%
\usepackage[numbers]{natbib}
\usepackage{filecontents}
\usepackage[latin9]{inputenc}
\usepackage[letterpaper]{geometry}
\usepackage{pifont}
\usepackage{float}
\usepackage{color}
\usepackage{adjustbox}
\usepackage{diagbox}
\usepackage{multirow}
\usepackage[makeroom]{cancel}
\usepackage{array}
\usepackage{mathrsfs}
\usepackage{tikz}
\usepackage{comment}
\usepackage{csquotes}

\graphicspath{{./Figures/}}

\usepackage{amsfonts}

\usepackage{color}
\usepackage{adjustbox}
\usepackage{diagbox}
\definecolor{black}{rgb}{0,0,0}

\definecolor{red}{rgb}{1,0,0}

\definecolor{blue}{rgb}{0,0,1}

\numberwithin{equation}{section}

\renewcommand{\div}{\mathop{\rm div}\nolimits}

\newcommand{\xoe}{{\bfa{x}\over\epsilon}}
\newcommand{\ep}{\epsilon}

\newcommand{\todo}[1]{{\color{red}{#1}}}
\newcommand{\jrp}[1]{{\color{orange}{#1}}}

\newcommand{\dx}{ \mathrm{d}x}
\newcommand{\dy}{ \mathrm{d}y}
\newcommand{\dt}{ \mathrm{d}t}

\newcommand{\beq}{\begin{equation}}
\newcommand{\eeq}{\end{equation}}
\newcommand{\beqq}{\begin{equation*}}
\newcommand{\eeqq}{\end{equation*}}
\newcommand{\beqas}{\begin{eqnarray*}}
\newcommand{\eeqas}{\end{eqnarray*}}
\newcommand{\bsp}{\begin{split}}
\newcommand{\eesp}{\end{split}}

\makeatother

\usepackage[english]{babel}

\author[J.S.R. P\MakeLowercase{ark}, S.W. C\MakeLowercase{heung},  T. M\MakeLowercase{ai}]
{J\MakeLowercase{un} S\MakeLowercase{ur} R\MakeLowercase{ichard} P\MakeLowercase{ark$^a$}, 
S\MakeLowercase{iu} W\MakeLowercase{un} C\MakeLowercase{heung$^b$}, 
 T\MakeLowercase{ina} M\MakeLowercase{ai$^{c,d,*}$}}

\date{\today}

\title[M\MakeLowercase{ultiscale simulations for multi-continuum} R\MakeLowercase{ichards equations}]
{\textsf{\LARGE M\MakeLowercase{ultiscale simulations for multi-continuum} R\MakeLowercase{ichards equations}}$^1$\footnote{$^1$A\MakeLowercase{ccepted manuscript by} J\MakeLowercase{ournal of} C\MakeLowercase{omputational and} A\MakeLowercase{pplied} M\MakeLowercase{athematics} (2021).  T\MakeLowercase{he doi of published journal article: \url{https://doi.org/10.1016/j.cam.2021.113648}}.}}

\begin{document}

\begin{abstract}
In this paper, we study a multiscale method for simulating a dual-continuum unsaturated flow problem within complex heterogeneous fractured porous media.  Mathematically, each of the dual continua is modeled by a multiscale Richards equation (for pressure head), and these equations are coupled to one another by transfer terms.  On its own, Richards equation is already a nonlinear partial differential equation, and it is exceedingly difficult to solve numerically due to the extra nonlinear dependencies involving the soil water.  To deal with multiple scales, our strategy is that starting from a microscopic scale, we upscale the coupled system of dual-continuum Richards equations via homogenization by the two-scale asymptotic expansion, to obtain a homogenized system, at an intermediate scale (level).  Based on a hierarchical approach, the homogenization's effective coefficients are computed through solving the arising cell problems.  To tackle the nonlinearity,  after time discretization, we use Picard iteration procedure for linearization of the homogenized Richards equations.  At each Picard iteration, some degree of multiscale still remains from the intermediate level, so we utilize the generalized multiscale finite element method (GMsFEM) combining with a multi-continuum approach, to upscale the homogenized system to a macroscopic (coarse-grid) level.  This scheme involves building uncoupled and coupled multiscale basis functions, which are used not only to construct coarse-grid solution approximation with high accuracy but also (with the coupled multiscale basis) to capture the interactions among continua.  These prospects and convergence are demonstrated by several numerical results for the proposed method. 
\end{abstract}

\maketitle


\vskip10pt

%





\vspace{8pt}

\noindent $^a$\textit{Jun Sur Richard Park}; Department of Mathematics, The University of Iowa,
Iowa City, IA, USA; junsur-park@uiowa.edu 

\vspace{5pt}

\noindent $^b$\textit{Siu Wun Cheung}; Center for Applied Scientific Computing, Lawrence
Livermore National Laboratory, Livermore, CA 94550, USA; cheung26@llnl.gov

\vspace{5pt}

\noindent $^{*}$Corresponding author: \textit{Tina Mai}; $^c$Institute of Research and Development, Duy Tan University, Da Nang, 550000, Vietnam; $^d$Faculty of Natural Sciences, Duy Tan University, Da Nang, 550000, Vietnam; maitina@duytan.edu.vn 



\newpage












\section{Introduction}\label{intro}


Soil moisture predictions have imperatively drawn attention not only in agriculture but also in hydrology, environment, energy balances and global climate forecast, etc. 
The involved processes are mathematically modeled by unsaturated flow, namely, Richards equation \cite{r0,richarde1,richarde2,richarde3,richardsreview}, which delineates the seepage
of water into some porous media having pores filled with water and air \cite{ry1}.  Moisture close to the soil surface is mainly affected by precipitation together with evaporation that are strongly coupled in nonlinear manners, which leads to our considering coupled Richards equations.  
It is noticed from \cite{richardsreview} that analytical solutions of Richards equation can be found only with restrictive assumptions, 
so most applied problems demand a numerical solution in one or two or three dimensions, especially in our setting of coupled equations.   One of the most interesting features of the Richards equation is that regardless of its simple
derivation, it is arguably one of the hardest equations to find 
reliable and accurate numerical solution in all of hydrosciences \cite{richardsreview}.  Moreover, the porous media can possess complex heterogeneous rock characteristics, intricate geography of
fractures, 
multi-continuum background, high contrast and multiple scales, etc.  

Among various challenges, the fractures' high permeability can significantly influence the fluid flow processes and raise a requirement for a specific strategy in order to build mathematical model and computational approaches.  One of the earliest strategies is the hierarchical model, which is employed to characterize the multiscale fractures interacting with porous materials \cite{mcontinua17}. Now, given at hand a multi-continuum strategy
\cite{baren,warren1963behavior,kazemi1976numerical,wu1988multiple,pruess1982fluid}, we can represent and analyze complicated processes with multiple scales of heterogeneity or fractures.     
Generally, in each continuum, different fluid flow models can be used. For example, in \cite{darcy-forchheimer}, the mathematical model allows
coupling Darcy--Forchheimer flow in the fractures with Darcy flow in the matrix.

In this work, we focus on a multi-continuum model for unsaturated flow, which arises from the coupled
system of nonlinear Richards equations, in complex multiscale fractured porous media. Our considering dual-continuum background thus consists of a system of small scale highly connected fractures (the so-called natural ones or well-developed ones or highly developed ones) as the first continuum and a matrix as the second continuum \cite{baren}.  In such heterogeneous multi-continuum media, the simulation of nonlinear fluid flow as coupled Richards equations is even more challenging, primarily because of the nonlinearity, different properties of continua, high contrast, multiple scales, and mass transfer among continua and a variety of scales in many forms. 

To solve those difficult problems, one needs some kind of model reduction for flow simulation.  Conventional methods involve dividing the considered domain into coarse-scale grid blocks, where effective properties in each coarse block are computed \cite{homod}.  This procedure in common upscaling methods based on homogenization employs the fine-scale solutions of local problems in every coarse block or representative volume.  Such computation, however, may not reveal multiple important modes in each coarse block including the continua's interaction.



That downside urged the multi-continuum strategies \cite{baren,arbogast20,warren1963behavior,kazemi1976numerical,wu1988multiple,pruess1982fluid} on coarse grid.  Physically, each continuum is considered as a system (over the whole domain) so that the flow between different continua can be conveniently described.  
In the fine grid, different continua are neighboring.  In the coarse grid, they co-exist (through mean characteristics \cite{baren}) at every point of the region, and they interact with one another.  Mathematically, a couple of equations are formed for each coarse block, and an individual equation corresponds to one of the dual continua on the fine grid.  For example, in fractured media, the flow equations for the system of natural fractures and the matrix are written distinctly with some exchange terms.  Such interaction terms are coupled toward a system of coupled multiscale equations following the mass conservation law.
In order to reach that goal, even when each continuum is not connected topologically, one can assume that it is connected to the other (throughout the type of the coupling and the whole domain), given that it has only global (non-local) effects. 



In these contexts, we now examine (in detail) our chosen dual-continuum background.  For simulating flow in naturally fissured rock, Barenblatt introduced the first dual-porosity model \cite{baren}.  He suggested in that work two continua to describe low and high porosity continua, that is, respectively system of small connected fractures and matrix, which are also used in our paper.  An example of some beginning work on dual continua based on \cite{baren} is \cite{arbogast20} (1990), where homogenization theory was employed, within a naturally fractured medium, to obtain a general form of the double-porosity model representing single-phase flow.  For each continuum, both inter and intraflow exchanges are taken into consideration. Basically, the dual-continuum background can be in any form, where the above schemes can be utilized.



For such homogenization theory but now with multiple scales, we contribute in this paper the detailed results (thanks to \cite{rh1,rh2}) on homogenization of the two-scale dual-continuum system of nonlinear Richards equations (which are linear in \cite{rh1,rh2}), with optimal computational cost.  As in \cite{rh2}, the interaction between the continua here is scaled as $1/ \e$ (which was $1/ \e^2$ in \cite{rh1}), and $\e$ stands for the periodically microscopic scale of the material.  This homogenized problem arrives
from the two-scale asymptotic expansion \cite{papa,bakhvalov1989homogenisation,jikov2012homogenization}.  
For the dual continua, we show that given this scale $1/ \e$ (as in \cite{rh2}) of the interaction term, the coupled
limits (which became only one equation in \cite{rh1} for the scale $1/ \e^2$) occur when letting the microscopic scale $\e$ vanish in the homogenization procedure.  The homogenization's effective coefficients are computed from solutions of four cell problems (as in \cite{rh2}).  Solving these local cell problems at every macroscopic point using the same fine mesh is expensive because there are a huge number of such points.

Therefore, our additional contribution is a development of the hierarchical technique from \cite{rh1} (which has not been investigated in \cite{rh2}) for solving that four cell problems using a reasonable number of degrees of freedom while the accuracy essentially remains.  More specifically, we solve the four cell problems using a dense network of macroscopic points.  This scheme takes advantage of the fact that there are similar characteristics of neighboring representative volume elements (RVEs) \cite{rh1, rh2, bridge20,mcl}, which leads to close effective properties of the neighboring RVEs.  Such a hierarchical technique achieves optimal computational complexity by employing different levels of resolution of FE spaces at different macroscopic points, for the four cell problems.  In particular, regarding the cell problems, solutions 
at the
points 
belonging to lower levels
in the hierarchy,
which are solved from higher levels of resolution,
are employed to correct 
each solution 
at a nearby macroscopic point
lying on a higher level in
the hierarchy,
which is obtained via a lower level of accuracy (resolution). 
This hierarchy of macrogrids of points corresponds to a nest of approximation spaces having different levels of resolution.  More specifically, each level of macroscopic points is appointed to an approximation finite element (FE) space, then these points as well as space are where one solves the cell problems.  We theoretically prove that this hierarchical FE approach reaches the same
level of accuracy as that of the full solve where cell problems at every macroscopic point are solved by FE space with the highest level of
resolution, but the required number of degrees
of freedom is optimal as expected. 

In the literature, for other multiscale equations, the hierarchical finite element (FE) scheme has been advanced to solve the cell problems and calculate the effective
coefficients. In \cite{brown13}, the approach was developed for the case of deterministic
two-scale Stokes--Darcy systems within a slowly varying porous material. Later in \cite{brown17}, the hierarchical strategy was applied to a two-scale
ergodic random homogenization problem without any assumption on microscopic periodicity. Based on these achievements, in our paper, we extend the hierarchical approach to the two-scale
dual-continuum nonlinear system where the interaction terms are scaled as $\mathcal{O}(1/ \e)$, with respect to the computation of homogenized coefficients.  
The interaction terms lead to interesting four
cell problems (as in \cite{rh2}), that is, a system of coupled four equations.

Starting from a microscopic scale $\e$, the resulting homogenized coupled Richards equations are still nonlinear at an intermediate scale.  To tackle the challenge from the nonlinearity, after discretization by time, we apply linearization via traditional Picard iteration (with a reasonable termination criterion) for each time step until the stopped time.  At the current iteration of the linearization, note that (as in \cite{mcl}) some degree of multiscale still remains in the homogenized equations, so direct method (to be discussed later) is not effective.  To overcome the difficulties from such multiple scales as well as high contrast, fractures, heterogeneity, 
and to reduce computational cost, we utilize our recent numerical strategy \cite{mcl} based on the GMsFEM \cite{G1,chung2016adaptive,cho2017generalized} combining with the dual-continuum approach, to attain coarse-grid (macroscopic) level of the coupled dual-continuum linearized homogenized equations.  Note that the fine-grid scale in our paper (as in \cite{mcl}) is the intermediate scale resulting from the homogenization procedure, so it is different from the fine-grid scale of the GMsFEM in \cite{mcontinua17} and \cite{Spiridonov2019}.

Normally, the GMsFEM helps us systematically create multiple multiscale basis functions, by including new basis functions (degrees of freedom) in every coarse block.  Such new basis functions are computed by establishing local snapshots and operating
local spectral decomposition in the snapshot space.  Hence, the obtained eigenfunctions can convey the local properties to the global ones, through the coarse-grid multiscale basis functions.  

It is important to note that the GMsFEM's convergence is related to the eigenvalue decay of the local spectral problems. To construct the global multiscale basis functions, we select certain number of eigenfunctions that correspond to the smallest eigenvalues from each coarse neighborhood. Then the error converges with a rate correlated with $1/\Lambda$, where $\Lambda$ is the minimum of the excluded eigenvalues over all the coarse neighborhoods. This suggests that one needs to include all the multiscale basis functions that correspond to the smallest eigenvalues for a good solution's accuracy (see \cite{chung2016adaptive}, for instance).

Efficiently, the GMsFEM has been utilized for a variety of multi-continuum problems.  Recently, there has been an example \cite{organic17} regarding shale gas transport in dual-continuum background comprising organic and inorganic media.  With this motivation, a third continuum can be included in dual continua toward triple continua (see \cite{Tony11}, for instance).  On the whole, flow simulation was investigated in heterogeneously varying multicontinua \cite{mcontinua17, Spiridonov2019, ericmultiporoelastic19a} and in fractured porous media \cite{mcontinua17,akkutlu2017,akkutlu-poro-25}. 

Before the advent of GMsFEM, several multiscale methods are developed to solve problems with such heterogeneous characteristics, for example, multiscale
finite element method (MsFEM) \cite{Ms,Msnon} (which the GMsFEM directly based on), multiscale finite volume method (MsFVM) \cite{msfvm}, and heterogeneous multiscale methods
(HMM) \cite{hmm}.  More specifically, in \cite{msfem-r16, msfem-r17, GintingThesis}, the authors establish the coarse-grid
approximation thanks to the MsFEM for solving the unsaturated flow problems possessing heterogeneous coefficients.  In \cite{homo1}, upscaling
method is utilized for the Richards equation.  Especially, multiscale methods for flow problems in fractured
porous media are addressed in \cite{akkutlu2017}.  

In light of the GMsFEM's success, there are numerous and vital studies on new model reduction techniques consisting of constraint energy minimizing (CEM) GMsFEM \cite{cem1,cem2f,cemnlporo,cem20} 
and involved numerical methods for multi-continuum systems with fractures \cite{cheung2018constraint} comprising non-local multi-continuum method (NLMC) \cite{vasilyeva2019nonlocal,chung2018non,cem28,nlmc14}.  In NLMC approach, one builds multiscale basis functions from the solutions of some local constrained energy minimizing problems as in the CEM-GMsFEM.  These strategies also efficaciously deal with multiscale as well as high-contrast components in multi-continuum fractured media.   


Besides such advanced multiscale methods, as mentioned above, traditional direct approach can tackle those difficulties in multi-continuum flow simulation, via fine-grid simulation, in a couple of steps.  First, a locally fine grid is created.  Second, one discretizes the flow equations on that fine grid and derives a  global solution from the set of local solutions.  This scheme can be operated under popular frameworks, for example, the Finite Element Method (FEM) in \cite{fem} and Finite Volume Method (FVM).  However, due to enormous size of computation, the direct numerical simulation of multiple-scale problems is hard even
with the assistance of supercomputers.  By parallel
computing (which utilizes domain decomposition
methods), this hardness can be eased a little \cite{Ms}.  
Classical parallel computing
methods, nevertheless, demand rigorous interaction among processes, 
whereas the size of the discrete problem still remains.  Hence, multiscale methods are needed,
to obtain the small-scale effect on the
large scales (without the desire of solving all the small-scale details) via constructing coarse-grid approximations through offline (precomputed)
multiscale basis functions. 







Given a time step and a Picard iteration of the linearization (of the above homogenized nonlinear equations), we present the GMsFEM \cite{G1,g37} 
for solution of the linearized homogenized equations from the nonlinear unsaturated multi-continuum flow problem in heterogeneous porous fractured materials. We base on the previous works for linear case \cite{mcl} and for saturated flow problem \cite{mcontinua17}. Herein, we respectively extend those work to fractured domain and study
solution of the unsaturated flow problem in two-dimensional case (while our technique can also work for three-dimensional case).  

%

Within the GMsFEM, we investigate two basis types: uncoupled and coupled multiscale basis.  In the first case (called uncoupled GMsFEM), multiscale basis functions will be established for each continuum independently, by taking into consideration only the permeability and ignoring the transfer terms.  We then employ the GMsFEM presented above.  In the second case (called coupled GMsFEM, which is focused in our paper), multiscale basis functions will be built by solving a coupled problem for snapshot space and operating a spectral decomposition. From this step, the GMsFEM is also applied.  Generally, in the GMsFEM, multiscale
basis functions are constructed in order to automatically locate each continuum through solution of the local spectral problems \cite{continuum38,Msnon,G1}. Now, for coupled
system of equations in multi-continuum models, we solve local coupled system of equations in order to establish highly accurate coupled multiscale basis
functions that describe complex interaction among continua in the coarse-grid level.
   

In numerical simulations, within each Picard iteration, we aim at coupling the GMsFEM with the multi-continuum approach.  To demonstrate accuracy and robustness of the proposed coupled GMsFEM, we consider a dual-continuum
background (connected fractures and matrix) model as above. Several numerical examples are presented for two-dimensional test problems.  
At the last time step and at the end of the Picard iteration procedure, the multiscale solution is compared with the reference fine-scale solution.  Our numerical results (after benefiting both coupled and uncoupled GMsFEM) prove that the GMsFEM solutions converge when we increase the number of local basis functions, and that coupled GMsFEM (for large interaction coefficients) has higher accuracy than uncoupled GMsFEM.  Also, our numerical results show that the GMsFEM is able to incorporate with the dual-continuum approach to reach an accurate solution 
via only few basis functions, and that the GMsFEM is robust respecting high-contrast coefficients.
Regarding further results about Picard iteration procedure for linearization of the coupled dual-continuum systems of Richards equations, numerically, convergence is guaranteed by a very small termination criterion number.  Theoretically, we also prove the global convergence of this Picard linearization process in Appendix \ref{cp}.

The paper is organized as follows.  In Section \ref{fus}, preliminaries are presented.  Section \ref{sysr} is about formulating the system of dual-continuum coupled nonlinear multiscale Richards equations and the corresponding homogenized system (that we will work with), in heterogeneous fractured porous media.  We provide in Section \ref{pre} fine-scale finite element discretization and Picard iteration procedure for linearization of the homogenized system.  In Section \ref{gms}, the GMsFEM is presented for our homogenized coupled nonlinear system, making use of 
both uncoupled and coupled GMsFEM.  We show in Section \ref{numer} several numerical results for this homogenized system.  Section \ref{discuss} is regarding open discussion.  Conclusions are congregated in Section \ref{conclude}.  In Appendix \ref{sec:appendix_homogenization}, we give a derivation of the homogenized equations (introduced in Section \ref{sysr}).  Appendix \ref{hier} is devoted to analyzing hierarchical numerical solutions of the cell problems from the homogenization.
Appendix \ref{proofthm} gives a proof of the main Theorem \ref{maintheorem} (about convergence results for the hierarchical solve).  In the last Appendix \ref{cp}, we present a proof for the global convergence of Picard linearization process.   

\section{Preliminaries}\label{fus}
Let $\Omega$ be our bounded computational domain in $\mathbb{R}^d$.  To ease our discussion, we consider $d=2$ in the remaining of the paper, 
but the method can be generalized to $d=3$. 
We refer the readers to \cite{C-G-K, gne, cemnlporo, mcl} for the basic preliminaries.  Latin indices $i,j$ are in the set $\{1,2\}$.  Functions are denoted by italic capitals (e.g., $f$), vector fields in $\mathbb{R}^2$ and $2 \times 2$ matrix fields over $\Omega$ are denoted by bold letters (e.g., $\bfa{v}$ and $\bfa{T}$).  The spaces of functions, vector fields in $\mathbb{R}^2$, and $2 \times 2$ matrix fields defined over $\Omega$ are respectively represented by italic capitals (e.g., $L^2(\Omega)$), 
boldface Roman capitals (e.g., $\bfa{V}$), 
and special Roman capitals (e.g., $\mathbb{S}$).   

Throughout this paper, the symbol $\nabla$ stands for the gradient with respect to $\bfa{x}$ of a function which only depends on the variable $\bfa{x}$ (or the variables $\bfa{x}$ and $t$).  By $\nabla_x\,,$ we represent the partial gradient with respect to $\bfa{x}$ of a function which depends on $\bfa{x},t$ and other variables as well.  Einstein summation is reflected by repeated indices.  The class $C^0$ consists of all continuous functions.  Spaces of periodic functions come with subscript $\#\,.$   

Consider the space $V: = H_0^1(\Omega) = W_0^{1,2}(\Omega)$.  Its dual space 
(also called the adjoint space) is denoted by 
$H^{-1}(\Omega)\,,$ which consists of continuous linear functionals on $H_0^1(\Omega)$.  The value of a functional $f \in H^{-1}(\Omega)$ at a point 
$v \in H_0^1(\Omega)$ is 
denoted by the inner product $\langle f,v \rangle $.  Whereas, the notation $(\cdot,\cdot)$ stands for the standard $L^2$ inner product.  

The Sobolev norm $\| \cdot \|_{W_0^{1,2}(\Omega)}$ is of the form 
\[\|v\|_{W_0^{1,2}(\Omega)} = \left(\|v\|^2_{L^2(\Omega)} + 
\|\nabla v \|^2_{\bfa{L}^2(\Omega)}\right)^{\frac{1}{2}}\,.\]
Here, $\| \nabla v \|_{\bfa{L}^2(\Omega)}:= \| | \nabla v | \|_{\bfa{L}^2(\Omega)}\,,$ where 
$| \nabla v|$ denotes the Euclidean norm of the $2$-component vector-valued function 
$ \nabla v$; and for $\bfa{v} = (v_1,v_2)$, $\| \nabla \bfa{v}\|_{\mathbb{L}^2(\Omega)}:= \| | \nabla \bfa{v}| \|_{\mathbb{L}^2(\Omega)}\,,$ where 
$| \nabla \bfa{v}|$ denotes the Frobenius norm of the $2 \times 2$ matrix $\nabla \bfa{v}$.  
We recall that the Frobenius norm on $\mathbb{L}^2(\Omega)$ is defined by 
$| \bfa{A} |^2 : = \bfa{A} \cdot \bfa{A} = \text{tr}(\bfa{A}^{\text{T}} \bfa{A})\,.$

The dual norm to $\| \cdot \|_{H_0^1(\Omega)}$ is $\| \cdot \|_{H^{-1}(\Omega)}$, that is,
\[\| f \|_{H^{-1}(\Omega)} = \sup_{v \in H_0^1(\Omega)}
\frac{|\langle f,v \rangle|}{\|v\|_{H_0^1(\Omega)}}\,.\]
 

For every $1 \leq r < 
\infty$, we use $L^r(0,T;X)$ to represent the Bochner space \cite{evans} with the norms 
\[\|\phi\|_{L^r(0,T;X)} := \left(\int_0^T \|\phi \|_{X}^r \, \dt\right)^{1/r} < + \infty\,, 
\]
\[\|\phi\|_{L^{\infty}(0,T;X)}: = \sup_{0 \leq t \leq T} \|\phi\|_{X}  < + \infty\,,\]
where $(X, \| \cdot \|_{X})$ is a Banach space, for example $X=H_0^1(\Omega) \,.$  Also, we define 
\[H^1(0,T;X):= \left \{ \phi \in L^2(0,T;X) \, : \, \partial_t \phi \in L^2(0,T;X) \right \}\,.\]
To reduce notation \cite{lporo},
instead of $L^2(0,T;H_0^1(\Omega))$, we denote by $V=H_0^1(\Omega)$ the space for the pressure head $p_i(t,\cdot)$ and by $\bfa{V} = V \times V = H_0^1(\Omega) \times H_0^1(\Omega)$ the space for $\bfa{p}(t,\cdot)=(p_1(t,\cdot),p_2(t,\cdot))$, where $i=1,2$ and $t \in [0,T], T > 0\,.$  

\section{System of coupled Richards equations and its homogenization}\label{sysr}

As in Section \ref{intro}, let $\ep$ be the characteristic length representing the periodically small scale variability of the media.   In the current section, the system of Richards equations is based on \cite{Spiridonov2019,homo1,nlnlmc31}.  We first give an overview of the system then focus on the $1/ \e$-scale case in Subsection \ref{1e}.  

%
%
The following system of Richards equations is considered (by combining the mass conservation law with Darcy's law, respectively):
\beq\label{L12}
\bsp
&\frac{\partial \Theta^{\e}_1(p_1^\ep(t,\bfa{x}))}{\partial t} + \div (\bfa{v}_1^\ep)+L_{12}^\ep = f_1(t,\bfa{x}) \quad \text{in } (0,t) \times \Omega\,,\\
&\frac{\partial \Theta^{\e}_2(p_2^\ep(t,\bfa{x}))}{\partial t} + \div (\bfa{v}_2^\ep)- L_{21}^\ep = f_2(t,\bfa{x})\quad \text{in } (0,t) \times \Omega\,,\\
& \bfa{v}_1^\ep = -k_1^\ep(\bfa{x},p_1^\ep)\nabla p_1^\ep(t,\bfa{x}) \hspace{75pt} \quad \text{in } (0,t) \times \Omega\,,\\
& \bfa{v}_2^\ep = -k_2^\ep(\bfa{x},p_2^\ep)\nabla p_2^\ep(t,\bfa{x}) \hspace{75pt} \quad \text{in } (0,t) \times \Omega\,,
\end{split}
\eeq
where each $p_i^\ep=p_i^\ep(t,\bfa{x})$ [$m$] denotes the pressure head, for $i=1,2$ (each $i$ corresponds to a continuum); $\bfa{v}_i^\ep$ [$m s^{-1}$] represents the Darcy velocity;
$\Theta^{\e}_i(p_i^\ep(t,\bfa{x}))$ is volumetric soil water content; $f_i(t, \bfa{x})[s^{-1}]$ refers to source or sink term; ${\bfa{K}}_i^{\ep}(\bfa{x},p_i^{\ep})[ms^{-1}]$ stands for unsaturated hydraulic conductivity tensor \cite{nlnlmc31, Spiridonov2019}, which is bounded, 
symmetric and positive definite.  
In our paper, the media are simply assumed to be isotropic 
so that each hydraulic conductivity $\bfa{K}_i^{\ep}(\bfa{x},p_i^{\ep})$ becomes function $k^{\e}_i(\bfa{x},p_i^{\ep})$ multiplying with the identity matrix $\bfa{I}$ \cite{kanhe}. 
The techniques here can be extended to anisotropic media.  We use $L^{\e}_{12}$ and $L^{\e}_{21}$ to denote the transfer terms between fracture-matrix and matrix-fracture, respectively, where  
\[\int_{\Omega} L_{12}^\ep \, \dx - \int_{\Omega} L_{21}^\ep \, \dx = 0\,, \quad \text{and }
L_{12}^{\e} = L_{21}^{\e} \approx Q_{12}^{\e}(\bfa{x},p_1^{\e},p_2^{\e})(p_1^{\e} - p_2^{\e})\,,\]
with the mass transfer (exchange) term \cite{Spiridonov2019} 
\begin{equation}
 \label{q12}
 Q_{12}^{\e} = Q_{21}^{\e}= Q_1^{\e} = Q_2^{\e}=Q^{\e}\,,
\end{equation}
that are calculated via the lower unsaturated hydraulic conductivity 
\begin{equation}
 \label{shape}
 Q^{\e}_1 = Q^{\e}_2 = \zeta k^{\e}_2(\bfa{x}, p_2^{\e})\,,
\end{equation}
having $\zeta$ as the shape factor \cite{warren1963behavior,baren}.

%
%
%

After substituting the Darcy's law (the last two equations of \eqref{L12}) and the mass transfer term $Q^{\e}_{12}$ at \eqref{q12} into the mass conservation equations (the first two equations of \eqref{L12}), we get the following system for the pressure heads: 
\beq\label{Q12}
\bsp
&\frac{\partial \Theta^{\e}_1(p_1^\ep)}{\partial t} - \div (k_1^\ep(\bfa{x},p_1^\ep)\nabla p_1^\ep)+ Q_1^\ep (\bfa{x},p_1^\ep,p_2^\ep)(p_1^\ep-p_2^\ep )= f_1\quad \text{in } (0,t) \times \Omega\,,\\
&\frac{\partial \Theta^{\e}_2(p_2^\ep)}{\partial t} - \div (k_2^\ep(\bfa{x},p_2^\ep)\nabla p_2^\ep) + Q_2^\ep (\bfa{x},p_1^\ep,p_2^\ep)(p_2^\ep-p_1^\ep)= f_2\quad \text{in } (0,t) \times \Omega\,.
\end{split}
\eeq
Note that we use only one domain $\Omega$ here, as in \cite{ericmultiporoelastic19a,mcl,rh1,rh2}.

%


In Section \ref{numer}, the unsaturated hydraulic conductivity is written as a product of permeability (depending on the spatial variable) and some function (depending on the pressure head variable) \cite{Spiridonov2019,nlnlmc31}.
Note that hydraulic conductivity \cite{rkk} is the fluid's ability to pass through material, while permeability \cite{rkk} is the material's ability allowing fluid to pass through.  Permeability is a feature of the medium itself, whereas hydraulic conductivity is the characteristic of both medium and fluid \cite{rkk}.  Also note that generally, intrinsic permeability or simply permeability $k$ (temperature independent)
does not coincide with saturated hydraulic conductivity $K_s$ (temperature dependent).

The source or sink term $f_i(t, \bfa{x})$ can be
potential recharge flux at the ground surface, 
evaporation, (plant) transpiration, precipitation rate \cite{egsource}, leakage through a confining layer, rainfall, pumpage \cite{sinkpos}, surface infiltration, evapotranspiration, and runoff \cite{egsource}, etc.  Positive values of $f_i$ represent a source function, while negative values of $f_i$ imply a sink function \cite{sinkpos}.  
Some physical meanings of these $f_i$ terms \cite{meansource} are as follows.  The unsaturated flow (Richards equation) needs to be equipped with a term representing water uptake by plant roots whenever soil-water-plant system is considered. This term (as a function) is referred to as the sink term because
water in the soil is withdrawn by the root system.  
When water is supplemented to the soil, by an aquifer for instance, then the equipped term (as a function) is referred to as the source term. 


The form of $f_i$ can depend on the pressure heads $p_i^{\e}$ \cite{egsource}.  
For example, $f_1 = k_1 p_1^{\e} - k_2 p_2^{\e}$, $f_2 = -k_1 p_1^{\e} + k_2 p_2^{\e}$, such that $f_1 + f_2 = 0\,.$  Both $f_i$ can be simultaneously positive as in \cite{sources2pos}, 
also see \cite{egsource} for coupled surface and subsurface flows.  Each $f_i$ can be either positive or negative in a region.  For example, in the \textit{vadose} zone, we consider a source term
possessing both positive and negative values given by $f = f(\bfa{x}, z) = 0.006 \tu{cos}(4/3\pi z) \tu{sin}(2\pi \bfa{x})$ on $\Omega_{\tu{vad}}$, whereas it holds that
$f \equiv 0$ in the \textit{saturated} zone $\Omega_{\tu{gw}}$ \cite{LPicard}.





Some explicit examples of source or sink term are from \cite{source01, source06} as follows.  The first form of sink term is from \cite{source01}:
\begin{equation}\label{s01}
 f_1(p_1^{\e}) = \alpha(p_1^{\e})S_{\tu{max}}\,,
\end{equation}
where $\alpha(p_1^{\e})$ is a dimensionless soil water availability factor of pressure head $p_1^{\e}$, and $S_{\tu{max}}$ is the maximum possible root-water extraction when soil water is not restricting (that is, the potential transpiration).  The second example of sink term is from \cite{source06}, as follows:  
\begin{equation}\label{s06}
 f_1(\bfa{x}, p_1^{\e}) = \alpha_1(\bfa{x},p_1^{\e}) \alpha_2(\bfa{x},p_1^{\e}) g(\bfa{x}) S\,,
\end{equation}
where $S$ is the potential transpiration rate, $g(\bfa{x})$ is the root density function, $\alpha_1(\bfa{x}, p_1^{\e})$ represents the compensation mechanism, $\alpha_2(\bfa{x}, p_1^{\e})$ stands for water stress (for interpretation of each component, see \cite{source06}, for instance).




\bigskip

\subsection{The $1/\ep$-scale}\label{1e}
This Subsection is based on \cite{rh2}.  Recall from Section \ref{fus} that $\Omega$ is the bounded computational domain in $\mathbb{R}^2$ and Latin indices $i,j$ vary in the set $\{1,2\}$.
We let $Y$ be a unit cube in $\mathbb{R}^2$ and $[a,b]$ be an interval in $\mathbb{R}\,.$  
With $p_i^{\e} \in V = H_0^1(\Omega)$ having values in $[a,b]$, we define the following two-scale coefficients:
\begin{equation}\label{wx}
k_i^\ep(\bfa{x},p_i^{\e}) = k_i\left(\bfa{x},\frac{\bfa{x}}{\e}, p_i^{\e}\right), \, Q_i^\ep(\bfa{x},p_1^{\e},p_2^{\e}) = Q_i\left(\bfa{x},\frac{\bfa{x}}{\e},p_1^{\e},p_2^{\e}\right)\,.
\end{equation}
Note that these two-scale coefficients can be written in the following form (without $\bfa{x}$):
\begin{equation}\label{nox}
k_i^\ep(\bfa{x},p_i^{\e}) = k_i\left(\frac{\bfa{x}}{\e}, p_i^{\e}\right), \, Q_i^\ep(\bfa{x},p_1^{\e},p_2^{\e}) = Q_i\left(\frac{\bfa{x}}{\e},p_1^{\e},p_2^{\e}\right)\,,
\end{equation}
whenever our discussion does not involve $\bfa{x}$ alone.  Here, we assume that  $Q_i(\bfa{y},p_1^{\e},p_2^{\e}) \in C^0(Y \times [a,b] \times [a,b])$ and $k_i(\bfa{y},p_i^{\e}) \in C^0(Y \times [a,b])$ are continuous functions, which are $Y$-periodic with respect to $\bfa{y} = \dfrac{\bfa{x}}{\e}\,.$  Also, each $f_i(t, \cdot)$ is assumed to be in $L^2(\Omega)\,.$ 

Now, we consider the following problem:
\beq
\label{eq:original0}
\bsp
\frac{\partial p_1^\ep(t,\bfa{x})}{\partial t} - \div [k_1^\ep(\bfa{x},p_1^\ep(t,\bfa{x}))\nabla p_1^\ep(t,\bfa{x})]+ \frac{1}{\ep}Q_1^\ep(\bfa{x},p_1^\ep(t,\bfa{x})&, p_2^\ep(t,\bfa{x}))(p_1^\ep(t,\bfa{x})-p_2 ^\ep(t,\bfa{x}))\\
& = f_1(t,\bfa{x}) \  \textrm{in} \ (0,T) \times \Omega\,,\\
\frac{\partial p_2^\ep(t,\bfa{x})}{\partial t} - \div [k_2^\ep(\bfa{x},p_2^\ep(t,\bfa{x}))\nabla p_2^\ep(t,\bfa{x})] +\frac{1}{\ep} Q_2^\ep(\bfa{x},p_1^\ep(t,\bfa{x})&, p_2^\ep(t,\bfa{x}))(p_2^\ep(t,\bfa{x})-p_1^\ep(t,\bfa{x}))\\
& = f_2(t, \bfa{x}) \ \textrm{in} \ (0,T) \times \Omega\,.
\end{split}
\eeq
For $i,j=1,2\,,$ this system \eqref{eq:original0} is equivalent to the following equations:
\beq
\label{eq:original}
\bsp
\frac{\partial p_i^\ep(t,\bfa{x})}{\partial t} - \div (k_i^\ep(\bfa{x},p_i^\ep)\nabla p_i^\ep)+ \frac{1}{\ep} \sum_j Q_i^\ep(\bfa{x},p_1^\ep, p_2^\ep)(p_i^\ep-p_j ^\ep) = f_i(t,\bfa{x}) \  \textrm{in} \ (0,T) \times \Omega\,.
\end{split}
\end{equation}  
The system of equations \eqref{eq:original} is equipped with the initial condition $p_1^\epsilon(0,\bfa{x})= 0$, $p_2^\epsilon(0,\bfa{x})= 0$ in $\Omega\,,$ and with the Dirichlet boundary condition $p_1^\epsilon(t,\bfa{x})=p_2^\epsilon(t,\bfa{x})=0$ on $(0,T) \times \partial \Omega\,.$ 
 
Note that in \eqref{Q12}, the volumetric water content $\Theta^{\e}_i(p^{\e}_i)$ is usually a nonlinear function of the pressure head $p^{\e}_i\,,$ and  the following form is also valid \cite{richarde1}:
\[\frac{\partial \Theta^{\e}_i(p_i^\ep)}{\partial t} = C(p^{\e}_i) \frac{\partial p^{\e}_i}{\partial t}\,.\]
In our analysis, the homogenization much involves the nonlinear hydraulic conductivity $k_i^{\e}(\bfa{x}, p_i^{\e})\,.$  Whereas, the term $C(p^{\e}_i)$ is not important and can be ignored so that $\Theta^{\e}_i$ is the identity function, that is, $\Theta^{\e}_i(p_i^\ep) = p_i^\ep$ in \eqref{eq:original0}.
 
\bigskip 
 
 Toward homogenization of the original system \eqref{eq:original}, we postulate that the system's solutions (that is, the pressure heads $p_1^\ep$ and $p_2^\ep$) admit the following two-scale asymptotic expansion (see \cite{2scale_nl_expand}, 
 for instance):
\beq\label{eq:asymp1}
\bsp
p_1^\ep(t,\bfa{x}) = p_{10}\left(t,\bfa{x},\xoe\right) + \ep p_{11}\left(t,\bfa{x},\xoe\right)+\ep^2 p_{12}\left(t,\bfa{x},\xoe\right)+ \cdots\,,\\
p_2^\ep(t,\bfa{x}) = p_{20}\left(t,\bfa{x},\xoe\right) + \ep p_{21}\left(t,\bfa{x},\xoe\right)+\ep^2 p_{22}\left(t,\bfa{x},\xoe\right)+ \cdots\,,
\end{split}
\eeq
in $\Omega\,.$
For $i=1,2$, we define
\[p_i^{\e} = p_i^\e (t,\bfa{x})
:=\hat{p}_i\left(t,\bfa{x}, \dfrac{\bfa{x}}{\e}\right) = \hat{p}_i\left(t,\bfa{x}, \bfa{y}\right)
\,, \quad  
\text{with } \bfa{y} = \frac{\bfa{x}}{\e}\,.\] 
In order to process further, we denote 
\begin{equation}\label{u1p10}
p_1 = p_1(t,\bfa{x}) := p_{10}(t,\bfa{x}) = p_{10}(t,\bfa{x}, \bfa{y})\,, \quad p_2=p_2(t,\bfa{x}) := p_{20}(t,\bfa{x}) = p_{20}(t,\bfa{x}, \bfa{y})\,,
\end{equation}
where the fact that $p_{10}(t,\bfa{x}, \bfa{y})$ and $p_{20}(t,\bfa{x}, \bfa{y})$ are independent of $\bfa{y}$ can be easily verified as in \cite{rh2}.
 
\bigskip 
 
Assume that the mass transfer term, the hydraulic conductivity and its spatial gradient are uniformly bounded, that is, there exist positive constants $\underline{k}, \overline{k}$ and $\underline{m},\overline{m}$ such that the following inequalities hold: 
 \begin{align}
 \label{Coercivity}
 \begin{split}
\underline{k} \leq k_1(\bfa{y}, p_1), \ k_2(\bfa{y}, p_2), \ | \nabla_y k_1(\bfa{y},p_1)|, \ |\nabla_y k_2(\bfa{y},p_2)| \leq \overline{k}\,,\\
\underline{m}  \leq Q_1(\bfa{y},p_1,p_2), \ Q_2(\bfa{y},p_1,p_2) \leq \overline{m}\,.
\end{split}
\end{align}
Resulting from the original system \eqref{eq:original}, our homogenized system is \eqref{eq:homogenized}.  The details of the homogenization's derivation and the hierarchical algorithm for the cell problems can be found in Appendices \ref{sec:appendix_homogenization} and \ref{hier} (with a proof \ref{proofthm}), respectively.

\bigskip

Now, we consider $p_1 = p_1(t,\bfa{x})= p_{10}(t,\bfa{x}), \ p_2 = p_2(t,\bfa{x}) = p_{20}(t,\bfa{x})$ in $V= H^1_0(\Omega)$ as from \eqref{u1p10}, the homogenized conductivities $\kappa_1(p_1), \kappa_2(p_2)$ abbreviated for $\kappa_1(\bfa{x},p_1), \kappa_2(\bfa{x},p_2)$
in \eqref{star}, and $f_1 (t,\cdot), \ f_2(t,\cdot)$ in $L^2(\Omega)$ as from \eqref{eq:original}.  

Thanks to \cite{thomo}, one needs to solve the reduced form of the system \eqref{eq:homogenized} for $p_1(t,\cdot),p_2(t,\cdot) \in V\,,$ that is, in the domain $(0,T) \times \Omega$ as follows:
\begin{align}\label{r0}
\begin{split}
 \frac{\partial p_1}{\partial t} - \ddiv(\kappa_1(p_1) \, \nabla p_1) + \bfa{b}_{11}(p_1,p_2) \cdot \nabla p_1 +\bfa{b}_{12}(p_1,p_2) \cdot \nabla p_2 
 + c_1(p_1,p_2) \, (p_1 - p_2) = f_1\,,\\
 \frac{\partial p_2}{\partial t} - \ddiv(\kappa_2(p_2) \, \nabla p_2) + \bfa{b}_{21}(p_1,p_2) \cdot \nabla p_1 + \bfa{b}_{22}(p_1,p_2) \cdot \nabla p_2 + c_2(p_1,p_2) \, (p_2 - p_1) = f_2\,,
 \end{split}
\end{align}
equivalently (for $i,j=1,2$),
\begin{align}\label{r1}
\begin{split}
 \frac{\partial p_i}{\partial t} - \ddiv(\kappa_i(p_i) \, \nabla p_i) + \sum_j [\bfa{b}_{ij}(p_1,p_2) \cdot \nabla p_j ]
 + \sum_j [c_i(p_1,p_2) \, (p_i - p_j)] = f_i\,,
 %
 \end{split}
\end{align}
with the initial condition $p_i(0,\bfa{x})= 0$ in $\Omega\,,$ and with the Dirichlet boundary condition $p_i(t,\bfa{x})=0$ on $(0,T) \times \partial \Omega\,.$
Note that in the system \eqref{r0} and \eqref{r1}, each $c_i(p_1,p_2)$ (abbreviated for $c_i(\bfa{x},p_1,p_2)$) is mass transfer term and each $\bfa{b}_{ij}(p_1,p_2)$ (abbreviated for $\bfa{b}_{ij}(\bfa{x},p_1,p_2)$) is  nonlinear velocity (see \cite{thomo}, for instance), which are explicitly defined in \eqref{eq:homogenized} in terms of correctors
being solutions of
the so-called cell problems \eqref{eq:nj}--\eqref{eq:mj}.




\vspace{12pt}

\section{Fine-scale discretization and Picard iteration for linearization}
\label{pre}

This section is based on \cite{rpicardc, Spiridonov2019, cemnlporo}.  Note that there are a variety of recently developed linearization techniques for Richards equation, such as the Newton method, the modified Picard algorithm \cite{richarde1}, the L-scheme \cite{LPicard}, and a new multilevel Picard
iteration \cite{mlp}.  However, we will use the classical Picard linearization algorithm as in \cite{classicPicard}.  

We are given an initial pair $\bfa{p}_0 = (p_{1,0},p_{2,0}) \in \bfa{V}\,.$  With $i,j=1,2\,,$ on each continuum $i\,,$ providing a fixed $u_i \in V\,,$ and letting $\bfa{u} = (u_1,u_2) \in \bfa{V}\,,$ we define the following bilinear forms:  for $p, \phi \in V$ and $\bfa{p}=(p_1,p_2) \in \bfa{V}\,,$
\begin{align}
 a_i(p,\phi;u_i)&=\int_{\Omega} \kappa_i(u_i)\nabla p \cdot \nabla \phi \, \dx\,, \label{ai}\\
 b_i(\bfa{p},\phi;\bfa{u})&=\sum_j \int_{\Omega} (\bfa{b}_{ij}(u_1,u_2) \cdot \nabla p_j) \phi \, \dx\,, \label{bi}\\
 q_i(\bfa{p},\phi;\bfa{u})&= \sum_j \int_{\Omega} c_i(u_1,u_2)(p_i-p_j) \phi \, \dx\,. \label{qi}
\end{align}
Recall that $(\cdot,\cdot)$ denotes the standard $L^2(\Omega)$ inner product.  The variational form of \eqref{r1} is as follows: for $i=1,2$, find $\bfa{p}=(p_1,p_2) \in \bfa{V}$ such that
\begin{align}\label{r1e}
 \left(\frac{\partial{p_i}}{\partial t} , \phi_i \right) + a_i(p_i,\phi_i;p_i) + b_i(\bfa{p},\phi_i;\bfa{p}) + q_i(\bfa{p},\phi_i;\bfa{p}) = (f_i,\phi_i)\,,
\end{align}
with all $\bfa{\phi} = (\phi_1,\phi_2) \in \bfa{V}\,,$ for a.e.\ $t \in (0,T)\,.$  
For simplicity, we can drop the subscript $i$ while
keeping the meaning that each equation \eqref{r1e} corresponds to one of the dual continua.

To reach the first goal regarding time discretization (see \cite{rpicardc, Spiridonov2019}, for instance) of \eqref{r1e}, we will apply the following standard backward Euler finite-difference scheme: find $\bfa{p}=(p_1,p_2) \in \bfa{V}$ such that 
for any $\bfa{\phi}=(\phi_1,\phi_2) \in \bfa{V}\,,$
\begin{align}\label{r1ed}
 \left(\frac{p_{i,s+1} - p_{i,s}}{\tau} , \phi_i \right) + a_i(p_{i,s+1},\phi_i;p_{i,s+1}) + b_i(\bfa{p}_{s+1},\phi_i;\bfa{p}_{s+1}) &+ q_i(\bfa{p}_{s+1},\phi_i;\bfa{p}_{s+1})\\
 &= (f_{i,s+1},\phi_i)\,, \nonumber
\end{align}
where the time range $[0,T]$ is divided 
into $S$ equal intervals, with the time step size $\tau = T/S > 0$, 
and 
the subscript $s$ denotes the evaluation of a function at 
the time instant $t_s= s\tau$ (for $s=0,1,\cdots,S$).

Next, the nonlinearity in space will be handled through linearization using Picard iteration (see \cite{rpicardc, Spiridonov2019,cemnlporo}, for instance) as follows.  At the $(s+1)$th time step, we guess $\bfa{p}^0_{s+1} \in \bfa{V}\,.$  For $n=0,1,2, \cdots,$ given $\bfa{p}^n_{s+1} \in \bfa{V} \,$, we find $\bfa{p}^{n+1}_{s+1} \in \bfa{V}$ such that for any $\bfa{\phi}=(\phi_1,\phi_2) \in \bfa{V}\,,$ 
\begin{align}\label{r1el}
\left(\frac{p^{n+1}_{i,s+1} - p_{i,s}}{\tau} , \phi_i \right) + a_i(p^{n+1}_{i,s+1},\phi_i;p^n_{i,s+1}) + b_i(\bfa{p}^{n+1}_{s+1},\phi_i;\bfa{p}^n_{s+1}) &+ q_i(\bfa{p}^{n+1}_{s+1},\phi_i;\bfa{p}^n_{s+1})\\
&= (f_{i,s+1},\phi_i)\,. \nonumber
\end{align}
%
%
The Picard iteration process converges to a limit as $n \to \infty$ (see a theoretical proof in Appendix \ref{cp}).   
%
In practice, we terminate this process at an $\alpha$th iteration when it meets a certain stopping criterion, 
and let 
\begin{equation}\label{pdata}
 \bfa{p}_{s+1} = \bfa{p}_{s+1}^{\alpha}
\end{equation}
be the previous time data to proceed to the next temporal step in \eqref{r1ed}.  Throughout this work, we propose a stopping criterion using the relative successive difference, that is, given a user-defined tolerance $\delta_0 > 0$, if 
\begin{equation}\label{pt}
\dfrac{\|p_{i,s+1}^{n+1} - p_{i,s+1}^{n} \|_{L^2(\Omega)}}{\| p_{i,s+1}^{n} \|_{L^2(\Omega)}} \leq \delta_0\,,
\end{equation}
for both $i=1,2\,,$ then we terminate the iteration procedure.

Now, we discuss the fine-grid notation.  Toward discretizing the variational problem \eqref{r1e}, we first let $\mathcal{T}_h$ be a fine grid with grid size $h$.  Here, $h$ is assumed to be significantly small so that the fine-scale solution 
is close enough to the exact solution.  Second, we let $V_h$ be the conforming piecewise bilinear finite element basis space with reference to the rectangular fine grid $\mathcal{T}_h$, that is,
\begin{equation}\label{Vh}
V_h:= \{ u \in V: u|_K \in \mathcal{Q}_1(K) \; \forall K \in \mathcal{T}_h\}\,,
\end{equation}
where $\mathcal{Q}_1(K)$ is the space of all bilinear (or multilinear if $d>2$) 
elements on $K$.  We let $\bfa{V}_h = V_h \times V_h\,.$

In $\mathcal{T}_h$, the fully Picard discrete scheme reads: starting with an initial $\bfa{p}_{h,0} \in \bfa{V}_h$, 
at the $(s+1)$th time step, we guess $\bfa{p}^0_{h,s+1}\in \bfa{V}_h \,,$ 
and iterate in $\bfa{V}_h$ from \eqref{r1el}:
\begin{align}\label{r1elh}
\begin{split}
\left(\frac{p^{n+1}_{i,h,s+1} - p_{i,h,s}}{\tau} , \phi_i \right) + a_i(p^{n+1}_{i,h,s+1},\phi_i;p^n_{i,h,s+1}) + b_i(\bfa{p}^{n+1}_{h,s+1},\phi_i;\bfa{p}^n_{h,s+1}) &+ q_i(\bfa{p}^{n+1}_{h,s+1},\phi_i;\bfa{p}^n_{h,s+1}) \\
&= (f_{i,s+1},\phi_i)\,,
\end{split}
\end{align}
with any $\bfa{\phi}= (\phi_1,\phi_2) \in \bfa{V}_h$, for $n=0,1,2, \cdots\,,$ until meeting \eqref{pt} at some $\alpha$th Picard step.  
We use \eqref{pdata} for setting the previous time data $\bfa{p}_{h,s+1} = \bfa{p}_{h,s+1}^{\alpha}$ to go ahead to the next time step in \eqref{r1ed}.

\section{GMsFEM for coupled dual-continuum nonlinear equations}\label{gms}


\subsection{Overview}\label{gmsover}

The purpose of this section is to build multiscale spaces (in the pressure head computation) for the coupled nonlinear system \eqref{r1e}.  Toward establishing an appropriate generalized multiscale finite element method (GMsFEM, \cite{G1}), from the linearized system \eqref{r1el} in Section \ref{pre}, the nonlinearity may be treated as constant at each Picard iteration step (after temporal discretization) so that multiscale spaces are able to be built respecting this nonlinearity.      

First, we discuss the coarse-grid notation. Let $\mathcal{T}^H$ be a coarse grid that has $\mathcal{T}_h$ as a refinement.  We denote by $H$ the coarse-mesh size (with $h \ll H$).  Each element of $\mathcal{T}^H$ is named a coarse-grid block (or patch or element).  
We call $N$ the total number of coarse blocks (elements) and $N_v$ the total number of interior vertices of 
$\mathcal{T}^H$.  Let $\{\bfa{x}_j\}^{N_v}_{j=1}$ be the collection of vertices (nodes) in $\mathcal{T}^H\,.$    
The $j$th coarse neighborhood of the coarse node $\bfa{x}_j$ is defined by the union of all coarse elements $K_m \in \mathcal{T}^H$ possessing such $\bfa{x}_j\,,$ as follows:
\beq
\omega_j = \bigcup \{K_m \in \mathcal{T}^H : \bfa{x}_j \in \overline{K_m}\}\,.
\eeq  

Our primary goal is using the GMsFEM to seek a multiscale solution $\bfa{p}_{\tu{ms}}$ (which is a good approximation of the fine-scale solution 
$\bfa{p}_h$). 
In order to do so, first, we utilize (on coarse grid) the GMsFEM \cite{G1}, where we solve local problems (to be specified then) in each coarse neighborhood, to systematically construct multiscale basis functions (degrees of freedom for the solution) that still contain fine-scale information.  The resulting multiscale space is called  the global offline space ${\bfa V}_{\tu{ms}}\,,$ comprising multiscale basis functions.  Last, we find the multiscale solution ${\bfa p}_{\tu{ms}}$ in ${\bfa V}_{\tu{ms}}\,.$  For the GMsFEM, in this section, we note that the system (\ref{r1}) 
possesses multiscale high-contrast coefficients $\kappa_i, {\bfa b}_{ij}, c_i$ (depending on $\bfa{x}$) for $i,j = 1,2\,.$
We refer the readers to \cite{G1} for the GMsFEM's details and to \cite{mcl,gne, mcontinua17} for its overview.  In the following subsections, we will present two different methods (mainly based on \cite{mcl}) for constructing uncoupled multiscale basis functions (uncoupled GMsFEM) and coupled multiscale basis functions (coupled GMsFEM).  For each type of these methods, using the GMsFEM's framework \cite{G1} as above, we build a local snapshot space for each coarse neighborhood $\omega_j$ then solve a relevant local spectral problem (defined on the snapshot space), to generate a multiscale (offline) space $\bfa{V}_{\tu{ms}}\,.$  

More specifically, in the next Subsections \ref{uncg} and \ref{cg}, provided $\bfa{p}_{\tu{ms},s}$ (at time step $s$th) and $\bfa{p}^{n}_{\tu{ms},s+1}$ (at time step $(s+1)$th and Picard iteration $n$th), we will build $\bfa{V}_{\tu{ms}}=\bfa{V}^n_{\tu{ms},s+1}\,.$  In practice (Subsection \ref{alg}), given $\bfa{p}_{\tu{ms},0}$ and a starting guess $\bfa{p}^{0}_{\tu{ms},1}$, we will need to establish only one $\bfa{V}_{\tu{ms}} = \bfa{V}^0_{\tu{ms},1}\,.$  Here, the snapshot functions as well as the basis functions are time-independent.


\subsection{Uncoupled GMsFEM}\label{uncg}
This terminology means that multiscale basis functions in $V^i_{\tu{ms}}$ 
are constructed for the solutions $p^{n+1}_{i,\tu{ms},s+1}$, separately, with $i,j=1,2\,.$  

We let the fine-scale approximation (FEM) space for the $i$th continuum be $V^i_h(\omega_j)=V_h(\omega_j)$, which is the conforming space $V^i_h =V_h$ restricted to the coarse neighborhood $\omega_j\,.$  The notation $J_h(\omega_j)$ stands for the set of all nodes of the fine grid $\mathcal{T}_h$ locating on $\partial \omega_j$.  The cardinality of $J_h(\omega_j)$ is abbreviated by $N_{J_j}\,,$ and we let the index $k$ varies $1 \leq k \leq N_{J_j}\,.$  

We will construct multiscale basis functions for each $i$th continuum distinctly by excluding the transfer functions and taking into consideration only the conductivity $\kappa_i\,.$  In particular, for each $i$th continuum, on every coarse neighborhood $\omega_j$, we first solve the following local snapshot problem: find the $k$th snapshot function $\phi_{k,i}^{(j),\tu{snap}} \in V_h(\omega_j)$ satisfying
\beq\label{snapu}
\bsp
-\div ( \kappa_i(\bfa{x},p^n_{i,\tu{ms},s+1})\nabla \phi_{k,i}^{(j),\tu{snap}} ) &= 0 \ \ \ \text{in} \ \omega_j,\\
\phi_{k,i}^{(j),\tu{snap}} &= \delta_{k,i} \ \ \ \text{on} \ \partial \omega_j\,,
\end{split}
\eeq
where $\delta_{k,i}$ is a function defined as
\[\delta_{k,i}(\bfa{x}^j_m) =
 \begin{cases}
  1 \quad m = k\,,\\
  0 \quad m \ne k\,,
 \end{cases}
\]
for all $\bfa{x}^j_m$ in $J_h(\omega_j)\,,$ $1 \leq k \leq N_{J_j}$.  
Hence, for the $i$th continuum, we obtain the $j$th local snapshot space 
\[V^i_{\tu{snap}}(\omega_j) = \text{span}\{ \phi_{k,i}^{(j),\tu{snap}} \, \bigr | \,  1 \leq k \leq N_{J_j} \}\,.\]

Now, let $N_{i}$ be the number of interior vertices of $\mathcal{T}^H$ on the $i$th continuum.  The $j$th local multiscale basis functions are built on $\omega_j$ with respect to the $i$th continuum, by solving the local spectral problems:  find the $k$th eigenfunction $\psi_{k,i}^{(j)} \in V^i_{\tu{snap}}(\omega_j)$ and its corresponding real eigenvalue $\lambda_{k,i}^{(j)}$ such that for all $\xi_i$ in $V^i_{\tu{snap}}(\omega_j)\,,$ 
\beq\label{eeunc}
a_i^{(j)}(\psi_{k,i}^{(j)},\xi_i) = \lambda_{k,i}^{(j)} s_i^{(j)}(\psi_{k,i}^{(j)},\xi_i)\,.
\eeq
For any $\phi, \psi \in V^i_{\tu{snap}}(\omega_j)\,,$ these operators are defined as follows \cite{gne,mcontinua17, mcl}:  
\beq\label{schi}
\bsp
a_i^{(j)}(\phi,\psi) = \int_{\omega_j} \kappa_i(\bfa{x},p^n_{i,\tu{ms},s+1}) \nabla \phi \cdot \nabla \psi\dx,\\
s_i^{(j)}(\phi,\psi) = \int_{\omega_j} \kappa_i(\bfa{x},p^n_{i,\tu{ms},s+1})\left( \sum_{l=1}^{N_{i}} |\nabla \chi_{l,i}|^2\right) \phi\psi  \, \dx\,,
\end{split}
\eeq
where each $\chi_{l,i}$ is a standard
multiscale finite element basis function in the $i$th continuum, for the coarse node $\bfa{x}_l$ (that is,
with linear boundary
conditions for cell problems \cite{pou}).  We note that $\{\chi_{l,i}\}_{l=1}^{N_{i}}$ is a set
of partition of unity functions (for $\mathcal{T}^H$) supported in the $i$th continuum.  



%
After arranging the eigenvalues $\lambda_{k,i}^{(j)}$ from (\ref{eeunc}) in ascending order, we take the first $L_{\omega_j}$ eigenfunctions, and they are still denoted as $\psi_{1,i}^{(j)}, \cdots, \psi_{L_{\omega_j},i}^{(j)},.$  Last, we define the $k$th multiscale basis function for the $i$th continuum on $\omega_j$ by
 \[\psi_{k,i}^{(j),\tu{ms}} = \chi_{j,i} \psi_{k,i}^{(j)}\,,\]
where $1 \leq k\leq L_{\omega_j}\,.$

Within the $i$th continuum, the local auxiliary offline multiscale space is defined by
\[V_{\tu{ms}}^{i}(\omega_j) = \text{span}\left \{\psi_{k,i}^{(j),\tu{ms}} \, \bigr | \, 1 \leq k\leq L_{\omega_j} \right \}\,.\]
We then define the global offline space 
\[V_{\tu{ms}}^i = \sum\limits_{j=1}^{N_{i}} V_{\tu{ms}}^{i}(\omega_j)= \text{span}\left \{\psi_{k,i}^{(j),\tu{ms}} \, \bigr | \, 1 \leq j \leq N_{i} \,, 1\leq k\leq L_{\omega_j} \right \}\,.\]
Finally, the multiscale space as the global offline space 
is defined by 
\[
\bfa{V}_{\tu{ms}} = V_{\tu{ms}}^{1}\times V_{\tu{ms}}^{2}\,,\]
which will be employed to find solution at the next $(n+1)$th Picard iteration. 
\subsection{Coupled GMsFEM}\label{cg}
In this section, we construct coupled multiscale basis functions in $\bfa{V}_{\tu{ms}}$
for the solution $\bfa{p}^{n+1}_{\tu{ms},s+1}=(p^{n+1}_{1,\tu{ms},s+1}, p^{n+1}_{2,\tu{ms},s+1})\,.$   

Here, regarding the coupled local snapshot problems, we will take into account the interaction terms $c_1$ and $c_2$ from (\ref{r0}).  For the GMsFEM analysis, the operators of eigenvalue problems should be symmetric.  Thus, we only choose the symmetric part of $c_1$ and $c_2$, namely, $c_s = (c_1+c_2)/2.$  

More specifically, for $i,j,r=1,2\,,$ on each coarse neighborhood $\omega_j$, we solve the local snapshot problem: find the snapshot functions $\bfa{\phi}_{k,r}^{(j),\tu{snap}} = \left (\phi_{k,1,r}^{(j),\tu{snap}},\phi_{k,2,r}^{(j),\tu{snap}}\right)$ in $\bfa{V}_h(\omega_j)=V_h(\omega_j)\times V_h(\omega_j)$ that satisfy
\beq\label{snapc}
\bsp
-\div \left( \kappa_1(\bfa{x},p^n_{1,\tu{ms},s+1}) \nabla \phi_{k,1,r}^{(j),\tu{snap}} \right) + c_s (\bfa{x},\bfa{p}^{n}_{\tu{ms},s+1})\left(\phi_{k,1,r}^{(j),\tu{snap}}-\phi_{k,2,r}^{(j),\tu{snap}}\right)= 0 \ \ \ \text{in} \ \omega_j,\\
-\div \left( \kappa_2(\bfa{x},p^n_{2,\tu{ms},s+1})\nabla \phi_{k,2,r}^{(j),\tu{snap}} \right) + c_s (\bfa{x},\bfa{p}^{n}_{\tu{ms},s+1})\left(\phi_{k,2,r}^{(j),\tu{snap}}-\phi_{k,1,r}^{(j),\tu{snap}}\right)= 0 \ \ \ \text{in} \ \omega_j,\\
\bfa{\phi}_{k,r}^{(j),\tu{snap}} = \bfa{\delta}_{k,r} \ \ \ \text{on} \ \partial \omega_j\,,
\end{split}
\eeq
for $1 \leq k \leq N_{J_j}$, where each $\bfa{\delta}_{k,r}$ is specified as
\beq
\bsp
\bfa{\delta}_{k,r}(\bfa{x}_m) = \delta_k(\bfa{x}_m) \bfa{e}_r, \ \ r = 1,2\,,
\end{split}
\eeq
with all $\bfa{x}_m$ in $J_h(\omega_j)$ and $\{\bfa{e}_r \, | \, r=1,2\}$ as a standard basis in $\mathbb{R}^2\,.$  
The local snapshot space is then of the form
\beq
\bfa{V}_{\tu{snap}}(\omega_j) = \text{span} \left \{ \bfa{\phi}_{k,r}^{(j),\tu{snap}} \, \bigr | \,  1\leq k \leq N_{J_j},\ r=1,2 \right \}\,.
\eeq

Now, we solve the local eigenvalue problems:  find the $k$th eigenfunction $\bfa{\psi}_{k}^{(j)}= \left(\psi_{k,1}^{(j)}, \psi_{k,2}^{(j)}\right) \in \bfa{V}_{\tu{snap}}(\omega_j)$ 
and its corresponding real eigenvalue $\lambda_{k}^{(j)}$ such that for all $\bfa{\xi} \in \bfa{V}_{\tu{snap}}(\omega_j)\,,$
\beq\label{eec}
a_{c_s}^{(j)}\left(\bfa{\psi}_{k}^{(j)},\bfa{\xi}\right) = \lambda_{k}^{(j)} s^{(j)}\left(\bfa{\psi}_{k}^{(j)},\bfa{\xi}\right)\,,
\eeq
Such operators are defined as follows \cite{gne,mcontinua17,mcl}:
\beq\label{pouu}
\bsp
a_{c_s}^{(j)}(\bfa{\phi},\bfa{\psi}) & = \sum_{i=1}^2 \int_{\omega_j} \kappa_i(\bfa{x},p^n_{i,\tu{ms},s+1}) \nabla \phi_i\cdot \nabla \psi_i \, \dx\,,\\
%
s^{(j)}(\bfa{\phi},\bfa{\psi}) & = \sum_{i=1}^2 s_i^{(j)}(\phi_i,\psi_i) = \sum_{i=1}^2 \int_{\omega_j} \kappa_i(\bfa{x},p^n_{i,\tu{ms},s+1}) \left( \sum_{l=1}^{N_{v}} |\nabla \chi_{l,i}|^2\right) \phi_i \psi_i \, \dx\,,
\end{split}
\eeq
for any $\bfa{\phi} = (\phi_1,\phi_2), \bfa{\psi}=(\psi_1,\psi_2) \in \bfa{V}_{\tu{snap}}(\omega_j)\,.$ Here, each $\chi_{l,i}$ is defined as in \eqref{schi}.
%

After sorting the eigenvalues $\lambda_{k}^{(j)}$ from (\ref{eec}) in rising up order, we choose the first smallest $L_{\omega_j}$ eigenfunctions and still call them $\bfa{\psi}_1^{(j)}, \cdots,\bfa{\psi}_{L_{w_j}}^{(j)}\,.$   
Now, the $k$th
multiscale basis functions for $\omega_j$ are defined by 
\[\bfa{\psi}_{k}^{(j),\tu{ms}} =  (\chi_{j,1} \, \psi_{k,1}^{(j)} \,, \chi_{j,2} \, \psi_{k,2}^{(j)})\,,\]
where each $\chi_{j,i}$ is defined as in \eqref{schi} and $1 \leq k \leq L_{w_j}\,.$
These functions form the local offline multiscale space 
\[\bfa{V}_{\tu{ms}}(\omega_j)= \text{span}\left \{\bfa{\psi}_{k}^{(j),\tu{ms}} \, \bigr | \, 1\leq k\leq L_{w_j} \right \}\,.\]
Last, we obtain the multiscale space (as the global offline space)
\[
\bfa{V}_{\tu{ms}} =\sum\limits_{j=1}^{N_v} \bfa{V}_{\tu{ms}}(\omega_j) = \text{span}\left \{\bfa{\psi}_{k}^{(j),\tu{ms}} \, \bigr | \, 1 \leq j \leq N_v \,, 1\leq k\leq L_{w_j} \right \}\,,\]
where we seek solution for the later $(n+1)$th Picard iteration. 

\begin{remark}\label{conv_Lamda}
Note that the local spectral problems, (\ref{eeunc}) and (\ref{eec}), are constructed taking into account the convergence analysis. The convergence rate of the proposed methods is proportional to $1/\Lambda$, where $\Lambda$ represents the minimum among all eigenvalues that correspond to eigenfunctions that are not included in the global offline space. This suggests that we include optimal number of multiscale basis functions associated with the smallest eigenvalues (see \cite{chung2016adaptive}, for instance).
\end{remark}


\subsection{GMsFEM for coupled nonlinear system of equations}\label{alg}

At the fixed time step $(s+1)$th, our tactic (as in \cite{gne, cemnlporo}) is using either the uncoupled or coupled GMsFEM (in Subsections \ref{uncg} or \ref{cg}) and the corresponding constructed offline multiscale space $\bfa{V}_{\tu{ms}}= \bfa{V}^0_{\tu{ms},1}$ (introduced at the end of Subsection \ref{gmsover}) to solve the problem (\ref{r1}) with equivalent variational form (\ref{r1e}) via linearization based on Picard iteration.  
During the online stage, the model reduction scheme reads: starting with $\bfa{p}_{\tu{ms},0} \in \bfa{V}_{\tu{ms}}$, 
at the $(s+1)$th time step, we guess $\bfa{p}^0_{\tu{ms},s+1}\in \bfa{V}_{\tu{ms}}$ 
and iterate in $\bfa{V}_{\tu{ms}}$ from \eqref{r1el}:
\begin{align}\label{ongmspicard}
\left(\frac{p^{n+1}_{i,\tu{ms},s+1} - p_{i,\tu{ms},s}}{\tau} , \phi_i \right) + a_i(p^{n+1}_{i,\tu{ms},s+1},\phi_i;p^n_{i,\tu{ms},s+1}) &+  b_i(\bfa{p}^{n+1}_{\tu{ms},s+1},\phi_i;\bfa{p}^n_{\tu{ms},s+1}) \nonumber \\
&+ q_i(\bfa{p}^{n+1}_{\tu{ms},s+1},\phi_i;\bfa{p}^n_{\tu{ms},s+1}) = (f_{i,s+1},\phi_i)\,,
\end{align}
with $\bfa{\phi} = (\phi_1,\phi_2)\in \bfa{V}_{\tu{ms}}$, for $n=0,1,2, \cdots\,,$ until reaching \eqref{pt} at some $\alpha$th Picard step.  
We employ \eqref{pdata} for choosing the previous time data $\bfa{p}_{\tu{ms},s+1} = \bfa{p}_{\tu{ms},s+1}^{\alpha}$ to advance on the next temporal step in \eqref{r1ed}.

\section{Numerical results}\label{numer}





In this section, we will show some numerical experiments to demonstrate the performance of our strategy, for both the uncoupled and coupled GMsFEM, with linearization by Picard iteration procedure.  In the simulations, we consider the high-contrast coefficients $a_i(\bfa{x})$ (for $i = 1,2$), which are presented in Fig.\ \ref{perm1}, having $\Omega = [0,1]^2\,.$  The blue regions are for $a_1(\bfa{x}) =10$, $a_2(\bfa{x}) =1$, whereas the yellow regions (channels) represent $a_1(\bfa{x}) = 10^5, a_2(\bfa{x}) = 10\,.$  Each coefficient $a_i(\bfa{x})$ is defined on $128 \times 128$ fine grid, while the coarse-grid size is $H=1/16\,.$  We choose the terminal time $T= S\tau = 2$, and the time step size is taken as $\tau = 1/10$.  Each initial pressure $p_i$ is zero.  The Picard iterative termination criterion is $\delta_0 = 10^{-5}$, which ensures the linearization process's convergence.

We assume that each medium is isotropic (and the anisotropic case is handled similarly).  Then, each hydraulic conductivity tensor becomes $\kappa_i$ multiplying with the identity matrix, for $i=1,2\,.$ 

\noindent \textbf{Example 1.}  We consider the following problem as a special case of \eqref{r0} (using the provided conditions there): find $\bfa{p}=(p_1, p_2) \in \bfa{V}$ in the domain $(0,T) \times \Omega$ such that 
\beq\label{nex}
\bsp
\frac{\partial p_1}{\partial t} - \div \left(\frac{a_1(\bfa{x})}{1+|p_1|} \nabla p_1 \right)+30{\bf b}_{11}\cdot \nabla p_1 + 30{\bf b}_{12}\cdot \nabla p_2 + \frac{10^5}{1+|p_1|}(p_1-p_2) = 1 \,,\\
\frac{\partial p_2}{\partial t} - \div \left(\frac{a_2(\bfa{x})}{1+|p_2|} \nabla p_2 \right)+30{\bf b}_{21}\cdot \nabla p_1 + 30{\bf b}_{22}\cdot \nabla p_2 + \frac{10^5}{1+ |p_2|}(p_2-p_1) = 1 \,,
\end{split}
\eeq
with the initial condition $p_i(0,\bfa{x})= 0$ in $\Omega\,,$ and with the Dirichlet boundary condition $p_i(t,\bfa{x})=0$ on $(0,T) \times \partial \Omega\,$.
The vector fields are given by ${\bf b}_{11}= {\bf b}_{21} = (p_1,p_1)$, ${\bf b}_{12}={\bf b}_{22} = (-p_2,-p_2)$, and the functions $a_1(\bfa{x})$, $a_2(\bfa{x})$ are high-contrast (multiscale) coefficients shown in Fig.~\ref{perm1}.  

This example has a physical meaning.  That is, for the given dual-continuum porous media, the first equation can describe a flow (the homogenized first Richards equation) in small highly connected fracture network as the first continuum, and the second equation can represent a flow (the homogenized second Richards equation) in the matrix as the second continuum \cite{Spiridonov2019}.  Our model problem \eqref{nex} can also be attained from upscaling of some highly heterogeneous media using Representative Volume Element (RVE) Approach \cite{bridge20}.  In that paper, the authors employ sub RVE scale to form a multi-continuum model (at fine-grid level), that is upscaled further.  In our context, we assume that the upscaled multi-continuum model can possess highly heterogeneous coefficients.  For example, with reference to the second continuum, the first continuum owns much larger $a_i(\bfa{x})$ in its channels (that are much bigger with respect to RVE scales).  More generally, with varying permeabilities $a_1$ and $a_2$, our proposed strategy can be still effective.  


At the last time step $S$th and at the last Picard iteration, the coupled GMsFEM solution $p_{i,\tu{ms},S}$ in \eqref{pdata} will be compared with the fine-scale FEM solution $p_{i,h,S}$, by the relative error formula in weighted $L^2(\Omega)$-norm:
\beq
\label{eq:relativeL2}
100\cdot\frac{\norm{p_{i,\tu{ms},S} - p_{i,h,S}}_{L^2(\Omega)}}{\norm{ p_{i,h,S}}_{L^2(\Omega)}}\,, \quad \text{for } i=1,2\,.
\eeq
Here, the reference solution $\bfa{p}_{h,S}$ is from \eqref{r1elh} in the fine grid $\bfa{V}_h$, and the multiscale solution $\bfa{p}_{\tu{ms},S}$ is from \eqref{ongmspicard} in the coarse grid $\bfa{V}_{\tu{ms}}\,.$

We denote by $DOF_{\tu{fine}}$ the number of total degrees of freedom (basis functions) for the fine-scale FEM.  Tables \ref{errors11} and \ref{errors12} respectively represent the errors obtained from the coupled and uncoupled GMsFEM with the permeability coefficients $a_i(\bfa{x})$ (see Figs.\ \ref{perm11} and \ref{perm12}). A fine-scale reference solution $p_{2,h,S}$ (solved via the FEM) is plotted ig.\ \ref{sol11}, whereas Fig.\ \ref{sol12} describes the coupled GMsFEM solution $p_{2,\tu{ms},S}\,.$ 

We tested the performance by using a variety of number of basis functions. Let $\tu{dim} (\bfa{V}_{\tu{ms}})$ denote the number of total degrees of freedom used for each GMsFEM implementation.  The results in Tables \ref{errors11} and \ref{errors12} demonstrate that our scheme is robust with respect to the contrast in the coefficients $a_i(\bfa{x})$ as well as able to give accurate approximation of solution with few local basis functions per each coarse neighborhood.  Also, from Tables \ref{errors11} and \ref{errors12}, we observe first that both the coupled and uncoupled GMsFEM solutions nicely converge as the number of local basis functions increases, and secondly, for large interaction coefficients, the coupled GMsFEM has higher accuracy than the uncoupled GMsFEM.  That second advantage is remarkable, as the GMsFEM involves the exchange terms in multiscale basis construction, while the uncoupled GMsFEM only takes into account the diffusion terms.

\bigskip

\begin{figure}[H]
	\centering
	\begin{subfigure}{0.45\textwidth}
  \includegraphics[width=\textwidth]{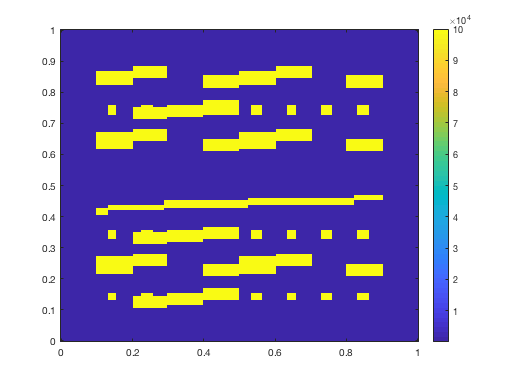}
  \caption{$a_1(\bfa{x})$. The value in each channel is $10^5$.}
  \label{perm11}
\end{subfigure}
\hfill
   \begin{subfigure}{0.45\textwidth}
  \includegraphics[width=\textwidth]{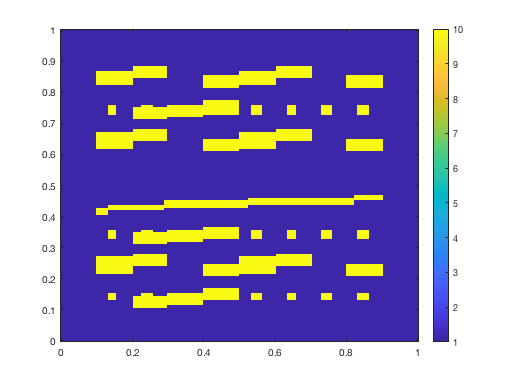}
  \caption{$a_2(\bfa{x})$. The value in each channel is 10.}
  \label{perm12}
 \end{subfigure}
 \caption{Coefficients $a_1(\bfa{x})$ and $a_2(\bfa{x})\,.$}
 \label{perm1}
\end{figure}

\begin{figure}[H]
	\centering
	\begin{subfigure}{0.45\textwidth}
  \includegraphics[width=\textwidth]{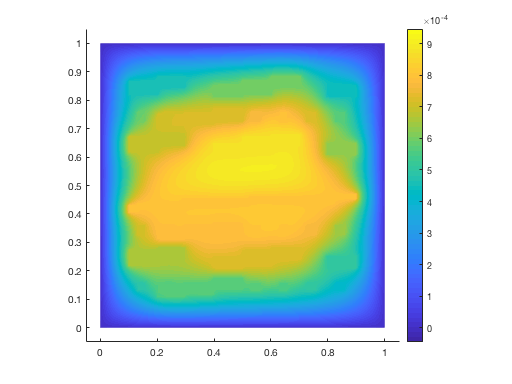}
  \caption{$p_{2,h,S}(T,\bfa{x})$, $DOF_{\tu{fine}} = 32258\,.$}
  \label{sol11}
\end{subfigure}
\hfill
   \begin{subfigure}{0.45\textwidth}
  \includegraphics[width=\textwidth]{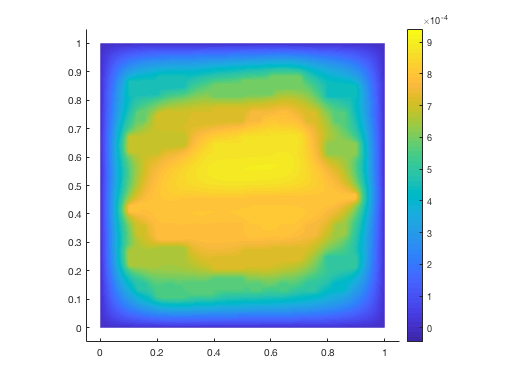}
  \caption{$p_{2,\tu{ms},S}(T,\bfa{x})$, dim($\bfa{V}_{\tu{ms}})=1800\,.$}
  \label{sol12}
 \end{subfigure}
 \caption{Solutions using FEM and Coupled GMsFEM.}
 \label{sol1}
\end{figure}

\begin{table}[h]
\label{Errors1}
\centering
\begin{subtable}{0.48\textwidth}
\begin{tabular}{|c|c|c|}
  \hline
 \multirow{2}{*}{dim($\bfa{V}_{\tu{ms}}$)} & $p_1$  & $p_2$ \\ 
 \cline{2-3} & $L^2$ Errors(\%) & $L^2$ Errors(\%)\\
  \hline \hline
 900 &  3.4208480& 3.56363346   \\
1800 & 0.56111391 & 0.70133747\\
[.2em]
2700 &  0.30925842  & 0.45617447 \\
[.2em]
3600& 0.18980716 & 0.33344175\\
[.2em]
4500 & 0.10368591 &0.23142539\\

  \hline
\end{tabular} \\
\caption{ Coupled GMsFEM, $DOF_{\tu{fine}}$ = 32258. 
} 
\label{errors11}
\end{subtable}
\begin{subtable}{0.48\textwidth}
\begin{tabular}{|c|c|c|}
  \hline
 \multirow{2}{*}{dim($\bfa{V}_{\tu{ms}}$)}& $p_1$  & $p_2$\\ 
 \cline{2-3}& $L^2$ Errors(\%) & $L^2$ Errors(\%)\\
  \hline \hline
900&  9.39526936&9.40019948\\
1800 & 3.42474881 & 3.42323362\\
[.2em]
2700 & 0.76386230& 0.76127447 \\
[.2em]
3600 &  0.56297485  & 0.56092131 \\
[.2em]
4500  & 0.37650901&0.37607187\\

  \hline
\end{tabular} \\
\caption{ Uncoupled GMsFEM, $DOF_{\tu{fine}}$ = 32258. 
} 
\label{errors12}
\end{subtable}
\caption{Errors.}
\end{table}

\noindent \textbf{Example 2.}  We consider a model using more complicated right hand side functions, including both sinks and sources with commonly used constitutive relationships between the water content and the pressure head, namely, van Genuchten-Mualem model (note that Gardner-Basha model also works) \cite{15mualemv, santos2006hydraulic}.  Using the van Genuchten-Mualem model in \cite{15mualemv}, we assume that its volumetric water content function is the identity, and we focus on the nonlinearity of the unsaturated hydraulic conductivity only. Our problem is to find $\bfa{p}=(p_1, p_2) \in \bfa{V}$ in the domain $(0,T) \times \Omega$ such that 
\beq\label{nex2}
\bsp
\frac{\partial p_1}{\partial t} - \div \left(a_1(\bfa{x})K_r(p_1) \nabla p_1 \right)+30{\bf b}_{11}\cdot \nabla p_1 + 30{\bf b}_{12}\cdot \nabla p_2 + \frac{10^2}{1+|p_1|}(p_1-p_2) = f_1(\bfa{x}) \,,\\
\frac{\partial p_2}{\partial t} - \div \left(a_2(\bfa{x}) K_r(p_2) \nabla p_2 \right)+30{\bf b}_{21}\cdot \nabla p_1 + 30{\bf b}_{22}\cdot \nabla p_2 + \frac{10^2}{1+ |p_2|}(p_2-p_1) = f_2(\bfa{x}) \,,
\end{split}
\eeq
%
with the initial condition $p_i(0,\bfa{x})= 0$ in $\Omega\,,$ and with the Dirichlet boundary condition $p_i(t,\bfa{x})=0$ on $(0,T) \times \partial \Omega\,$, where we let  $\Omega = [0,1]^2$. 
\noindent Here, the relative hydraulic conductivity $K_{r}$ has the form
\begin{align}\label{Kr}
\begin{split}
K_{r}(p) = \frac{\left [1- (\alpha_{v}|p|)^{n'-1} \left(1+(\alpha_{v}|p|)^{n'}\right)^{-m'}\right ]^2} {\left(1 + (\alpha_{v}|p|)^{n'}\right)^{\frac{m'}{2}}}\,,
\end{split}
\end{align}
where we assume that the geometric mean of $\alpha_v$ is $0.15$ m$^{-1}$, $n'=2, m' = 0.5\,.$
The high-contrast $a_i(\bfa{x})$ for $i = 1,2$ are presented in Fig.~\ref{perm2}. We have $a_1(\bfa{x}) =10$, $a_2(\bfa{x}) =10$ in the blue regions and $a_1(\bfa{x}) = 10^5, a_2(\bfa{x}) = 100$ in the yellow regions. The vector fields are ${\bf b}_{11}= {\bf b}_{21} = (p_1,p_1)$, ${\bf b}_{12}={\bf b}_{22} = (-p_2,-p_2)$, and the source and sink terms are given by $f_1(\bfa{x}) = e^{x_1+x_2}$, $f_2(\bfa{x}) = -e^{x_1+x_2}\,,$ respectively.

We again employ the GMsFEM coupled with the Picard iteration to find the solutions $p_{1,\tu{ms},S}$ and $p_{2,\tu{ms},S}$ at the last time step $S$.
At the $(s+1)$th time step, we choose the initial guess ($\bfa{p}^0_{\tu{ms},s+1}\in \bfa{V}_{\tu{ms}}$) for the Picard iteration to be the solution data from the previous time step, $\bfa{p}_{\tu{ms},s} = \bfa{p}_{\tu{ms},s}^{\alpha}$, where $\bfa{p}_{\tu{ms},s}^{\alpha}$ is the solution from the final Picard iteration step at the $s$th time step.  A fine-scale reference solution $p_{1,h,S}$ (obtained via the FEM) is plotted Fig.\ \ref{sol21}, while Fig.\ \ref{sol22} represents the coupled GMsFEM solution $p_{1,\tu{ms},S}\,.$

The relative $L^2$ errors (\ref{eq:relativeL2}) with respect to the total number of multiscale basis functions for both the coupled and uncoupled GMsFEM are presented in Tables \ref{errors21} and \ref{errors22}, where we observe that the errors decay as the number of degrees of freedom increases.  Both the methods require only small number of local basis functions to achieve high accuracy approximations to the solutions. The coupled GMsFEM has better accuracy than the uncoupled GMsFEM, and this is because the interaction between the two continua is disregarded in the construction of local multiscale basis of the uncoupled GMsFEM.


\begin{figure}[H]
	\centering
	\begin{subfigure}{0.45\textwidth}
  \includegraphics[width=\textwidth]{per_a_1.png}
  \caption{$a_1(\bfa{x})$. The value in each channel is $10^5$.}
  \label{perm21}
\end{subfigure}
\hfill
   \begin{subfigure}{0.45\textwidth}
  \includegraphics[width=\textwidth]{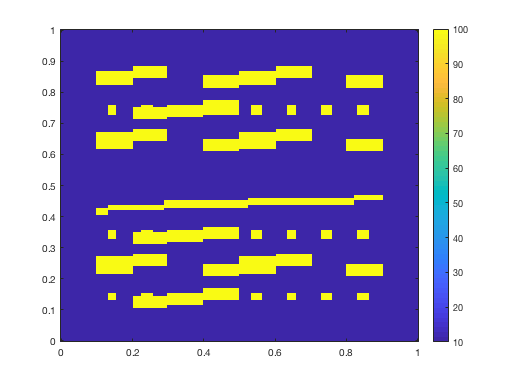}
  \caption{$a_2(\bfa{x})$. The value in each channel is 100.}
  \label{perm22}
 \end{subfigure}
 \caption{Coefficients $a_1(\bfa{x})$ and $a_2(\bfa{x})\,.$}
 \label{perm2}
\end{figure}

\begin{figure}[H]
	\centering
	\begin{subfigure}{0.45\textwidth}
  \includegraphics[width=\textwidth]{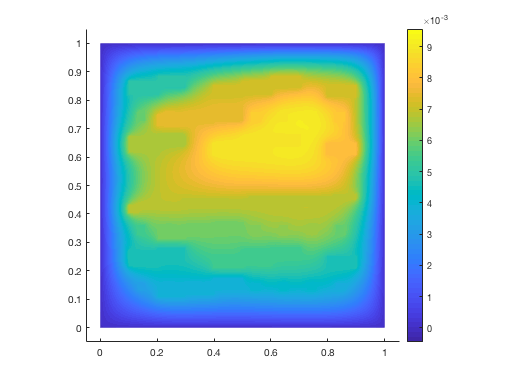}
  \caption{$p_{1,h,S}(T,\bfa{x})$, $DOF_{\tu{fine}} = 32258\,.$}
  \label{sol21}
\end{subfigure}
\hfill
   \begin{subfigure}{0.45\textwidth}
  \includegraphics[width=\textwidth]{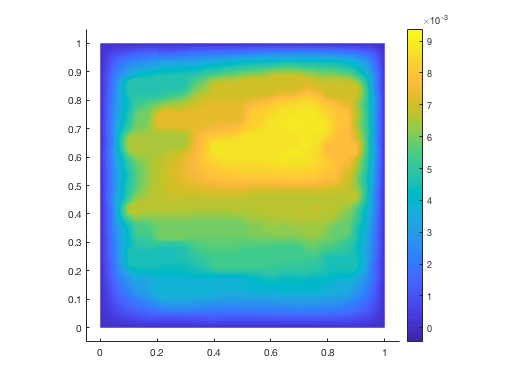}
  \caption{$p_{1,\tu{ms},S}(T,\bfa{x})$, dim($\bfa{V}_{\tu{ms}})=1800\,.$}
  \label{sol22}
 \end{subfigure}
 \caption{Solutions using FEM and Coupled GMsFEM.}
 \label{sol2}
\end{figure}

\begin{table}[h]
\label{Errors2}
\centering
\begin{subtable}{0.48\textwidth}
\begin{tabular}{|c|c|c|}
  \hline
 \multirow{2}{*}{dim($\bfa{V}_{\tu{ms}}$)} & $p_1$  & $p_2$ \\ 
 \cline{2-3} & $L^2$ Errors(\%) & $L^2$ Errors(\%)\\
  \hline \hline
 900 &  5.19096989& 5.62412231   \\
1800 & 2.52918581 & 2.01972957\\
[.2em]
2700 &  0.57498862  & 0.50771565 \\
[.2em]
3600& 0.43124964 & 0.38351447\\
[.2em]
4500 & 0.33847662 &0.26529760\\

  \hline
\end{tabular} \\
\caption{ Coupled GMsFEM, $DOF_{\tu{fine}}$ = 32258. 
} 
\label{errors21}
\end{subtable}
\begin{subtable}{0.48\textwidth}
\begin{tabular}{|c|c|c|}
  \hline
 \multirow{2}{*}{dim($\bfa{V}_{\tu{ms}}$)}& $p_1$  & $p_2$\\ 
 \cline{2-3}& $L^2$ Errors(\%) & $L^2$ Errors(\%)\\
  \hline \hline
900&  41.01167067&4.06627278\\
1800 & 4.59485651 & 4.84008130\\
[.2em]
2700 & 4.25388598& 2.44607018 \\
[.2em]
3600 &  2.40946791  & 1.96508453 \\
[.2em]
4500  & 0.92513418&0.70049912\\

  \hline
\end{tabular} \\
\caption{ Uncoupled GMsFEM, $DOF_{\tu{fine}}$ = 32258. 
} 
\label{errors22}
\end{subtable}
\caption{Errors.}
\end{table}

\bigskip

\section{Discussion}\label{discuss}


We note that our numerical strategy involves a dual-continuum generalized multiscale finite element method (GMsFEM) in dimension two.   Principally, it can be extended to multi-continuum systems, with discrete fractures (this type of fractures has not been considered in our paper).  In most cases, the online multiscale techniques can be utilized in order to obtain nicer results for nonlinear
problems \cite{online29nl,chung2015residual}. Moreover, one can create nonlinear basis functions for coarse-grid approximations of solutions with good accuracy \cite{nlnlmc31,nlnlmc32,nlnlmc33,nlbasis34,Msnon}.
Such multiscale basis functions as well as parallel computing can help establish approximations of solutions with higher accuracy, for complicated
processes in reality \cite{Spiridonov2019}.  
With nonstationary equations, including coupled Richards equations, recent splitting methods \cite{esplit20} would be interesting to apply.  Besides, discontinuous Galerkin (DG) method has been addressed in \cite{rdg}.  Regarding conservation laws, further hierarchical approaches could be useful for discontinuous Galerkin (DG) methods 
and even finite volume schemes \cite{hier1et, hier2et}.

%
%
%


\section{Conclusions}\label{conclude}
This paper investigated a multiscale method for simulations of the dual-continuum unsaturated flow problem modeled by coupled nonlinear Richards equations, in complex heterogeneous fractured porous media.  To handle the multiple scales, our approach is that starting from a microscopic scale, the coupled dual-continuum Richards equations are upscaled using homogenization via the two-scale asymptotic expansion, toward a system of coupled nonlinear homogenized equations, at an intermediate level of scale.  Utilizing a hierarchical finite element scheme, the homogenization's effective coefficients are calculated by solving the obtained cell problems.  To deal with the nonlinearity,  after temporal discretization, we employ the Picard iteration process for spatial linearization of the homogenized Richards equations.  Within each Picard iteration, some degree of multiscale still remains from the intermediate scale, so we apply the generalized multiscale finite element method (GMsFEM) incorporating with the coupled dual-continuum homogenized equations, in order to upscale the system to a macroscopic (coarse-grid) scale.  The GMsFEM's mission is to systematically create either uncoupled or coupled multiscale basis functions, independently for each equation (uncoupled GMsFEM), or commonly for the system (coupled GMsFEM).  Such multiscale basis functions can guarantee highly accurate coarse-grid approximation of the solution and clear representation of interactions among continua via the coupled GMsFEM.  These expectations and convergence are validated by several numerical results for the proposed method.  In addition, we theoretically proved the global convergence of the Picard iteration process in Appendix \ref{cp}.  Finally, we discussed some potential directions. 

\bigskip

\vspace{20pt}

\noindent \textbf{Acknowledgements.}  

This work was performed at Lawrence Livermore National Laboratory.
Lawrence Livermore National Laboratory is operated by Lawrence
Livermore National Security, LLC, for the U.S. Department of Energy,
National Nuclear Security Administration under Contract DE-AC52-07NA27344
and LLNL-JRNL-820660.

Tina Mai's research was funded by Vietnam National Foundation for Science and 
Technology Development (NAFOSTED) under grant number 101.99-2019.326, and her research was carried out in Vietnam.  

The authors thank the Reviewers very much
for their thoughtful comments, which have led to an improvement in the quality of the paper.


\bigskip

\vspace{20pt}

\noindent \textbf{Disclaimer.}
This document was prepared as an account of work sponsored by an agency of the
United States government.  Neither the United States government nor Lawrence
Livermore National Security, LLC, nor any of their employees makes any warranty,
expressed or implied, or assumes any legal liability or responsibility for the
accuracy, completeness, or usefulness of any information, apparatus, product, or
process disclosed, or represents that its use would not infringe privately owned
rights.  Reference herein to any specific commercial product, process, or
service by trade name, trademark, manufacturer, or otherwise does not
necessarily constitute or imply its endorsement, recommendation, or favoring by
the United States government or Lawrence Livermore National Security, LLC.  The
views and opinions of authors expressed herein do not necessarily state or
reflect those of the United States government or Lawrence Livermore National
Security, LLC, and shall not be used for advertising or product endorsement
purposes.


\bigskip

\vspace{90pt}

\appendix

\section{Derivation of homogenization}\label{sec:appendix_homogenization}
In this section, we derive the homogenized equations corresponding to our original equations (\ref{eq:original}) using the two-scale asymptotic expansion \eqref{eq:asymp1} from \cite{2scale_nl_expand}.  
Here, we have not rigorously proved the convergence.  However, basically, one can justify this asymptotic expansion under proper conditions and assuming sufficient regularity of the solutions of the two-scale model (see \cite{2scale_nl_expand}, for instance), or one can find some suitable topology so that the
solutions of the $\e$-problem \eqref{eq:original} converge \cite{ nguet1, allaire1, nguetseng2}, where $\e$ is the small scale.
Therefore, we can assume that those formal calculations can be validated.  

Regarding notation, as in \cite{pcell}, we let
\begin{equation}\label{H1p}
 H^1_{\#}(Y) = \{ g \in H^1_{\tu{loc}}(\mathbb{R}^2) \text{ such that $g$ is $Y$-periodic}\}
\end{equation}
be equipped with the norm $\| g \|_{H^1(Y)}\,.$  
We will need the quotient space 
\begin{equation}\label{qW}
 \mathcal{V}:= H^1_{\#}(Y) / \mathbb{R}
\end{equation}
defined as the space of classes of functions in $H^1_{\#}(Y)$ equaling up to an additive constant.  The following norm \cite{pcell} is considered in the space $\mathcal{V}$:
\[||\phi||_{\mathcal{V}} =  ||\nabla_y \phi||_{\bfa{L}^2(Y)}\,.\] 

Now, with $p_j \in V$, $\bfa{p}=(p_1,p_2) \in \bfa{V}$, and $N^i_j (\cdot,p_j)$, $M_j (\cdot,\bfa{p}) \in \mathcal{V}= H^1_{\#}(Y)/\mathbb{R}\,,$ for $i,j,\nu,m,r = 1,2$, and $\bfa{e}^i$ as the $i$th unit vector in the standard basis of $\mathbb{R}^d$ ($d=2$), applying the procedure in \cite{rh2} to our considered original system \eqref{eq:original}, 
we can find derivation of the homogenized equations, which are 
%
\begin{align}\label{eq:homogenized}
\begin{split}
\frac{\partial p_j}{\partial t} 
= & f_j + \div (\bfa{K}_j^*(p_j)\nabla p_j)\\
& + \sum_m \left \{ \div \left[ \left(\int_Y k_j(\bfa{y}, p_j) \nabla_y M_j(\bfa{y}, \bfa{p}) \, \dy\right)(p_m - p_j)\right] \right.\\
&\left. \hspace{46pt} +  (-1)^{j+m-1} \left(\int_Y Q_j(\bfa{y}, \bfa{p})N^i_m(\bfa{y}, p_m) \, \dy\right)\frac{\partial p_m}{\partial x_i} \right \}\\
 &+ \sum_m \left \{ -\left(\int_Y Q_j(\bfa{y}, \bfa{p}) (M_1(\bfa{y}, \bfa{p})+ M_2(\bfa{y}, \bfa{p}))  \, \dy\right)  \right.\\
 &\left.  \hspace{39pt} + \sum_r \left[\left(\int_Y\frac{\partial Q_j}{\partial p_r^{\e}} (\bfa{y}, \bfa{p}) \cdot N^i_r (\bfa{y}, p_r) \, \dy\right) \frac{\partial p_r}{\partial x_i} \right. \right.\\
& \left. \left. \hspace{73pt} + \sum_{\nu} \left(\int_Y\frac{\partial Q_j}{\partial p_r^{\e}} (\bfa{y}, \bfa{p}) \cdot M_r(\bfa{y}, \bfa{p})\, \dy\right)(p_{\nu} - p_r)\right] \right\} (p_m - p_j)\,,
\end{split}
\end{align} 
where (as in \cite{layer2, homo1},
for vectors
$\bfa{N}_1, \bfa{N}_2$) 
$\bfa{K}_m^*=(k^*_{mij})$, $k_m  = k_{mij} = k_{mir}\,,$
\begin{align}\label{eq:main10}
 \begin{split}
k^*_{mij}(p_m) &= \left(\int_Y  k_m(\bfa{y},p_m)\left(\bfa{I} + (\nabla_yN_m(\bfa{y},p_m))^{\tu{T}}\right) \, \dy\right)_{ij}\\
&= \int_Y  \left(k_m(\bfa{y},p_m) (\nabla_yN^j_m(\bfa{y},p_m) + \bfa{e}^j)\right)\cdot \bfa{e}^i \, \dy\\
&= \int_Y   \left(k_{mij} (\bfa{y},p_m) + \left(k_m(\bfa{y},p_m) \nabla_yN^j_m(\bfa{y},p_m)\right)\cdot \bfa{e}^i \right)\, \dy\\
&=\int_Y   \left(k_{mij} (\bfa{y},p_m) + k_{mir}(\bfa{y},p_m){\partial N^j_m(\bfa{y},p_m)\over \partial y_r}\right)\, \dy\\
&=\int_Y   k_{mir} (\bfa{y},p_m)\left( \delta_{rj}+ {\partial N^j_m(\bfa{y},p_m)\over \partial y_r}\right)\, \dy\,,
\end{split}
\end{align}
%
which are positive definite \cite{layer2,papa,rh1,rh2} and symmetric if $\bfa{K}_1,\bfa{K}_2$ are symmetric \cite{layer2}.  For isotropic media in this paper, we have 
\begin{equation}\label{star}
\bfa{K}_m^* = k_m^* \bfa{I}=\kappa_m \bfa{I}\,.
\end{equation}
Recalling \eqref{wx}--\eqref{nox}, we note that $k^*_{mij}(p_m)$ is abbreviated for $k^*_{mij}(\bfa{x},p_m)\,.$
%
Here, $N^i_j (\bfa{y},p_j)$, $M_j (\bfa{y},\bfa{p}) \in H^1_{\#}(Y)/\mathbb{R}$ \cite{pcell}
are the solutions of the four cell problems:
\begin{subequations}
\label{eq:cell}
\begin{align}
\label{eq:nj}
&\div_y [k_j(\bfa{y},p_j)(\bfa{e}^i + \nabla_y N^i_j(\bfa{y},p_j))] = 0\,,\\
%
\label{eq:mj}
&\div_y [k_j(\bfa{y},p_j)\nabla_y M_j(\bfa{y},\bfa{p})] + Q_j(\bfa{y},\bfa{p}) = 0\,,
%
%
\end{align}
\end{subequations} 
as in \cite{rh2}, with periodic boundary conditions, for $p_1,p_2 \in [a,b], \ \bfa{p} = (p_1,p_2) \in [a,b]^2$. 

\begin{remark}\label{intQ}
 Regarding each mass transfer term $Q_i > 0$ (for $i=1,2$), we assume that there is an average of it.  To guarantee the positivity, this average is large.  However, it can be absorbed into the
time, that is, in comparison with the time, this average is very small and can vanish when the time becomes very large or infinity.  That is, we assume
\beq
\label{eq:Qaverage}
\int_Y Q_i(\bfa{y},p_1,p_2) \, \dy=0\,.
\eeq
Such chosen $Q_i$ and its average as in \eqref{eq:Qaverage} guarantee the uniqueness of solution in $H^1_{\#}(Y) / \mathbb{R}$ for each of the problems \eqref{eq:mj}.
Also, it is easy to verify that each of the problems \eqref{eq:nj} has a unique solution in $H^1_{\#}(Y) / \mathbb{R}\,.$  
\end{remark}
%

%

\vspace{15pt}

The following assumption is about Lipschitz condition.

\begin{assumption}
\label{Lipschitz}
There is a positive constant such that for all $p_i$, $p_i'$ in $[a,b]$ and $\bfa{p}=(p_1,p_2)$, $\bfa{p}'=(p'_1,p'_2)$ in $[a,b]^2$, we have (as in \cite{rh1}, for $i=1,2$)
\begin{align}
\label{lipschitz}
\begin{split}
& \|k_i(\bfa{y},p_i) - k_i(\bfa{y},p_i')\|_{L^\infty(Y)}\le C|p_i-p_i'|\,, \quad 
\| \nabla_y (k_i(\bfa{y},p_i) - k_i(\bfa{y},p'_i))\|_{L^\infty(Y)} \le C|p_i-p'_i|\,,\\
& \|Q_i(\bfa{y},\bfa{p}) - Q_i(\bfa{y},\bfa{p}')\|_{L^\infty(Y)} \leq C|\bfa{p}-\bfa{p}'|\,.
\end{split}
\end{align}
\end{assumption}


\begin{remark}
 Crucially, the necessary condition for our proposed strategy to operate is that the above two-scale coefficients own Lipschitz (or H\"{o}lder) smoothness with respect to the macroscopic variable.  This assumption is judicious because the media's macroscopic characteristics generally vary smoothly \cite{rh1}.
\end{remark}


We can write the four cell problems (\ref{eq:cell}) in the variational form:  for $i,j=1,2\,,$ find $N^i_j (\cdot,p_j)$, $M_j (\cdot,\bfa{p})$ in $H^1_{\#}(Y)/\mathbb{R}$ such that for any $\phi_j,\psi_j \in H^1_{\#}(Y)\,,$
\begin{equation}
\label{eq:main11}
\begin{split}
\int_Y k_j(\bfa{y},p_j) \nabla_y N^i_j(\bfa{y},p_j) \cdot \nabla_y \phi_j(\bfa{y})\, \dy  
&= - \int_Y k_j(\bfa{y},p_j) \bfa{e}^i \cdot \nabla_y \phi_j(\bfa{y})\, \dy \,, \\
\int_Y k_j(\bfa{y},p_j) \nabla_y M_j(\bfa{y},\bfa{p}) \cdot \nabla_y \psi_j(\bfa{y})\, \dy  
&=  \int_Y Q_j(\bfa{y},\bfa{p}) \psi_j(\bfa{y})\, \dy \,.
\end{split}
\end{equation}

 
\begin{theorem}
\label{eu}
Each equation of (\ref{eq:main11}) has a unique solution in $\mathcal{V} = H^1_{\#}(Y) / \mathbb{R}\,.$
\end{theorem}
\begin{proof}
This is a standard result which follows from Remark \ref{intQ}, Lax-Milgram lemma and Assumption \ref{Lipschitz}.
\end{proof}
\section{Hierarchical numerical solutions of the cell problems}\label{hier}

We now present a development of the hierarchical algorithm introduced in \cite{brown13} and applied in \cite{rh1}. 
More specifically, we use the Galerkin finite element method (FEM) to approximate the solutions of the cell
problems (\ref{eq:cell}) at each macroscopic point 
in a hierarchy of macrogrids of such points, with a corresponding nest of FE solution spaces possessing different resolution levels.  The following steps form a general outline of such algorithm.

\textbf{Step 1:  Build nested FE solution spaces.}  Fixing the macropoints $p_j \in [a,b], \bfa{p} = (p_1,p_2) \in [a,b]^2$, $i,j=1,2$, we seek an approximation of each $N^i_j(\bfa{y},p_j), M_j(\bfa{y},\bfa{p})$ ($\in \mathcal{V}$) satisfying \eqref{eq:main11}, via the Galerkin FEM.  In order to do so, for some given positive integer $L\,,$ we construct nested FE solution spaces
\begin{equation}\label{fesps}
\mathcal{V}_{1} \subset \mathcal{V}_{2}\subset\cdots\subset {\mathcal V}_{L} \subset {\mathcal V} = H^1_{\#}(Y) / \mathbb{R}\,,
\end{equation}
and nested trial spaces
\begin{equation}\label{trysps}
\mathcal{W}_{1} \subset {\mathcal W}_{2}\subset\cdots\subset {\mathcal W}_{L} \subset {\mathcal W} = H^1_{\#}(Y)\,,
\end{equation}
where the integer indices $l = 1, 2, \dots, L$ denote the increasing resolution levels.  We choose a regularity space $\mathcal{U}$
(contained in $\mathcal{V}$) 
for the correct solutions $N_j^i(\bfa{y},p_j), M_j(\bfa{y},\bfa{p})$ (see \cite{reg-fem, brown13, guermond-reg}, for instance).  In our paper, $\mathcal{U} = H^2_\#(Y)$.  We can assume that the generated spaces $\mathcal{V}_k$ ($k = \overline{1,L}$) are Hilbert spaces of functions of $\bfa{y}$ in the cell domain $Y \in \mathbb{R}^2$ (see \cite{brown13}), and that for all $w \in \mathcal{U}$,
\begin{align}
\label{eq:FEapprox}
\begin{split}
\inf_{\phi\in{\mathcal V}_{L-l+1}}\|\nabla_y(w - \phi)\|_{\bfa{L}^2(Y)}
\le C \varkappa^{l-1} \eta \|w\|_{H^2(Y)}
\le C2^{-L+l-1}\|w\|_{H^2(Y)}\,,
\end{split}
\end{align}
where
the constant $C$ is independent of $L$ and $l\,,$ while $\varkappa = 2$ is the FE coarsening factor, and $\eta = 2^{-L}$ is the accuracy for the 
finest approximation spaces $(\mathcal{V}_L, \mathcal{W}_L)\,.$  

Note that for error estimate, we use the words ``small'' (``fine'') and ``large'' (``coarse'').  Whereas, for accuracy, we benefit the words ``high'' (``fine'') and ``low'' (``coarse'').

 
With the above established spaces $\mathcal{V}_k\,, \mathcal{W}_k$ and the regularity space $\mathcal{U}$ in $\mathcal{V}$, the errors between the continuous correct solutions $N_j^i(\bfa{y},p_j), M_j(\bfa{y},\bfa{p}) \in \mathcal{U}$ and the corresponding Galerkin FE approximations $\bar{N}^i_j(\bfa{y},p_j), \bar{M}_j(\bfa{y},\bfa{p})$ ($\in \mathcal{V}_{L-l+1}$) decreases (when $L-l+1$ increases) in a structured manner.  
That is, with test functions and $\phi \in \mathcal{W}_{L-l+1}$, we solve \eqref{eq:main11} for $\bar{N}^i_j(\bfa{y},p_j), \bar{M}_j(\bfa{y},\bfa{p})$ in $\mathcal{V}_{L-l+1}$ $\subset \mathcal{V}\,,$ satisfying the error conditions \eqref{eq:FEapprox}:  
%
%
\begin{align}
\label{eq:FEapprox1}
\begin{split}
||N_j^i(\bfa{y},p_j) - \bar{N}^i_j(\bfa{y},p_j)||_{\mathcal{V}}
&= \inf_{\phi\in{\mathcal V}_{L-l+1}}\|\nabla_y(N_j^i(\bfa{y},p_j) - \phi)\|_{\bfa{L}^2(Y)}
\le C \varkappa^{l-1} \eta \|N_j^i(\bfa{y},p_j)\|_{H^2(Y)}\\
& \le C2^{-L+l-1}\|N_j^i(\bfa{y},p_j)\|_{H^2(Y)}\,,\\
||M_j(\bfa{y},\bfa{p}) - \bar{M}_j(\bfa{y},\bfa{p})||_{\mathcal{V}} & \le C2^{-L+l-1}\|M_j(\bfa{y},\bfa{p})\|_{H^2(Y)}\,.
\end{split}
\end{align}
The largest (coarsest) error
is $\varkappa^{L-1} \eta = 2^{L-1} 2^{-L} = 2^{-1}$ (with $l=L$ in \eqref{eq:FEapprox1} for $\mathcal{V}_{L-l+1}= \mathcal{V}_{1}\,,$ the coarsest FE space).  


\textbf{Step 2:  Build hierarchy of macrogrids.}  With the given positive integer $L$, the following index sets are considered:\\
\beq
\bsp
I_{L} = \{m \in \mathbb{Z} : m = 0,1,2,3,\dots, 2^L \}, \
\bfa{I}_{L} = \{(m_1,m_2) \in \mathbb{Z}^2 : m_1, m_2 = 0,1,2,3,\dots,2^L \}.
\end{split}
\eeq
We define the sets $U_{1,L} \subset U_1 = [a,b],\ \bfa{U}_{2,L} \subset \bfa{U}_2 = [a,b]^2$ by
\begin{align}\label{du12}
 \begin{split}
U_{1,L} &= \left \{u_m = a + \frac{m}{2^L}(b-a), m \in I_{L}\right \} \,,\\
\ \bfa{U}_{2,L} &= \left \{\bfa{u_m} = \left(a+\frac{m_1}{2^L}(b-a),a+\frac{m_2}{2^L}(b-a)\right), \ \bfa{m} = (m_1, m_2) \in \bfa{I}_{L} \right \}\,.
\end{split}
\end{align}

Now, we construct hierarchies of the points in $U_{1,L}$, $\bfa{U}_{2,L}$.  Toward reaching that purpose, for each $l =1,2,\cdots L\,,$ we first define the following set $R_l \subset [a,b]$ (as in \cite{brown13}):
\beq
R_1 = \left \{a, a+\frac{b-a}{2}, b\right \}, \ R_l = \left \{a+\frac{(2r-1)(b-a)}{2^{l}}, r = 1,2,\dots, 2^{l-1}\right\} \text{ for } l \geq 2\,.
\eeq
Then, we define the hierarchy $\mathcal{S}_1^l$ of points in $U_{1,L}$ and the hierarchy $\bfa{\mathcal{S}}_2^l$ of points in $\bfa{U}_{2,L}$ as follows:
\beq\label{dS}
\mathcal{S}_1^l = R_l,\ \bfa{\mathcal{S}}_2^l = \left \{\bfa{u} = (u_1, u_2) : u_1, u_2 \in R_{l}\right \}\,, \text{ } l = \overline{1,L}\,.
\eeq
%

For example, with $[a,b] = [0,1]$, we have \cite{rh1}:
\[U_{1,1} = \left \{0, \frac{1}{2}, 1\right \}, \ U_{1,2} = \left \{0, \frac{1}{4}, \frac{2}{4}, \frac{3}{4}, 1\right \}, \ U_{1,3} = \left \{0,\frac{1}{8},  \cdots, \frac{8}{8}\right \}, \ U_{1,4} = \left \{0,\frac{1}{16},   \cdots, \frac{16}{16}\right\}\,,\]
and
\[\mathcal{S}_1^1 =  \left\{0, \frac{1}{2}, 1 \right \}, \quad \mathcal{S}_1^2 = \left \{\frac{1}{4}, \frac{3}{4} \right \}, \quad \mathcal{S}_1^3 = \left \{\frac{1}{8}, \frac{3}{8}, \frac{5}{8}, \frac{7}{8}\right \}, \quad \mathcal{S}_1^4 = \left \{\frac{1}{16}, \frac{3}{16}, \frac{5}{16}, \cdots, \frac{15}{16}  \right \}\,.\]

We note that $U_{1,1} = R_1=\mathcal{S}_1^1$ and 
\begin{equation}\label{dU}
U_{1,L} = \displaystyle\bigcup_{l=1}^{L} \mathcal{S}_1^l\,, \ \bfa{U}_{2,L} =\displaystyle\bigcup_{l=1}^{L} \bfa{\mathcal{S}}_2^l \,, 
\end{equation}
with $U_{1,1}\subset U_{1,2}\subset\cdots\subset U_{1,L} \subset [a,b]$, and $\bfa{U}_{2,1}\subset \bfa{U}_{2,2}\subset\cdots\subset \bfa{U}_{2,L} \subset [a,b]^2$.  


In this way, given $\mathcal{S}_1^1$ as the set of anchor points, we build a dense hierarchy of the macro-grid points as in \cite{brown13}. That is, 
for each point $p_i \in \mathcal{S}_1^l$ or $\bfa{p}=(p_1,p_2) \in \bfa{\mathcal{S}}_2^l$ (where $l \geq 2$), there exists at least one point from one of the previous levels, namely,
 $p_i' \in \bigcup_{k<l} \mathcal{S}_1^k$ or  $\bfa{p}' \in \bigcup_{k<l} \bfa{\mathcal{S}}_2^k$ such that dist($p_i$, $p_i'$) $<$ $\mathcal{O}(H 2^{-l})$ or dist($\bfa{p}$, $\bfa{p}'$) $<$ $\mathcal{O}(H 2^{-l})\,,$ for $i=1,2\,.$ 
 
\textbf{Step 3:  Calculating the correction term.}  \textit{Now, we effectively connect the nested FE spaces (of solutions) with the hierarchy of macrogrids (of points) in a computational approach}.  As in \cite{rh1}, for anchor points $p_i \in \mathcal{S}_1^1, \bfa{p} = (p_1,p_2) \in \bfa{\mathcal{S}}_2^1$ (these $\mathcal{S}_1^1\,, \bfa{\mathcal{S}}_2^1$ are the most sparse macrogrids), by the standard Galerkin FEM, we find the approximations $\bar{N}^i_j(\cdot,p_j)$, $\bar{M}_j(\cdot,\bfa{p}) \in {\mathcal V}_L$ (the finest FE space with highest accuracy) 
satisfying \eqref{eq:main11}:
\begin{equation}
\label{eq:main11db}
\begin{split}
\int_Y k_j(\bfa{y},p_j) \nabla_y \bar{N}^i_j(\bfa{y},p_j) \cdot \nabla_y \phi_j(\bfa{y})\, \dy  
&= - \int_Y k_j(\bfa{y},p_j) \bfa{e}^i \cdot \nabla_y \phi_j(\bfa{y})\, \dy \,, \\
\int_Y k_j(\bfa{y},p_j) \nabla_y \bar{M}_j(\bfa{y},\bfa{p}) \cdot \nabla_y \psi_j(\bfa{y})\, \dy  
&=  \int_Y Q_j(\bfa{y},\bfa{p}) \psi_j(\bfa{y})\, \dy \,, 
 %
\end{split}
\end{equation}
where $\phi_j,\psi_j \in \mathcal{W}_{L}$ (the finest FE space) and $i,j=1,2\,.$ 

Proceeding inductively, for $l \geq 2$, $p_i \in \mathcal{S}_1^l, \bfa{p} = (p_1,p_2) \in \bfa{\mathcal{S}}_2^l$, we choose the points  $p_i' \in (\bigcup_{l'<l} \mathcal{S}_1^{l'})$, $\bfa{p}' \in (\bigcup_{l'<l} {\bfa{\mathcal{S}}}_2^{l'})$ so that dist($p_i$, $p_i'$) and dist($\bfa{p}$, $\bfa{p}'$) are 
$\mathcal{O}(H2^{-l})$, where $l' \geq 1$ and $l',l \leq L\,.$  This is possible since we constructed the hierarchy
of grids (of macroscopic points) in a dense way (at the end of Step 2). We solve the following problems: find the correction terms (that are actually the terms to be corrected) ${\bar{N^{i}_j}}^c(\bfa{y},p_j)$, ${\bar{M_j^c}}(\bfa{y},\bfa{p})$ in ${\mathcal V}_{L+1-l}$ such that
%
\begin{equation*}
\label{eq:main52dbc} 
\begin{split}
&\int_Y k_j(\bfa{y},p_j) \nabla_y {\bar{N^i_j}}^c (\bfa{y},p_j) \cdot  \nabla_y \phi_j(\bfa{y}) \, \dy\\ &= 
-\int_Y (k_j(\bfa{y},p_j) - k_j(\bfa{y},p_j'))\nabla_y \bar{N}^{i}_j (\bfa{y},p_j') \cdot \nabla_y \phi_j(\bfa{y}) \, \dy\\
& \quad -\int_Y (k_j(\bfa{y},p_j) - k_j(\bfa{y},p_j')) \bfa{e}^i \cdot \nabla_y \phi_j(\bfa{y}) \, \dy\,,
\end{split}
\end{equation*}
\begin{equation}
\label{eq:main52'dbc} 
\begin{split}
\int_Y &k_j(\bfa{y},p_j) \nabla_y {\bar{M_j^c}}(\bfa{y},\bfa{p}) \cdot  \nabla_y \psi_j(\bfa{y}) \, \dy\\ = & 
-\int_Y (k_j(\bfa{y},p_j) - k_j(\bfa{y},p'_j))\nabla_y \bar{M}_j (\bfa{y},\bfa{p}') \cdot \nabla_y \psi_j(\bfa{y}) \, \dy \\
& -\int_Y (Q_j (\bfa{y},\bfa{p})-Q_j (\bfa{y},\bfa{p}'))\psi_j(\bfa{y}) \, \dy\,,
\end{split}
\end{equation}
for all $\phi_j,\psi_{j} \in \mathcal{W}_{L+1-l}\,.$  Note that the right-hand side data of \eqref{eq:main52'dbc} is all known because we have already found inductively the solutions $\bar{N^i_j}(\bfa{y},p_j')$, $\bar{M_j}(\bfa{y},\bfa{p}')$ (the 
hierarchical macro-grid interpolations of Galerkin FE approximations for $N^i_j(\bfa{y},p_j')$, $M_j(\bfa{y},\bfa{p}')$ at \eqref{eq:main11} in $\mathcal{V}$) 
from (\ref{eq:main11db}) in finer FE spaces $\mathcal{V}_{L-l'+1}$ ($\supset \mathcal{V}_{L-l+1}$) at macro-grid points $p_j' \in \bigcup_{l'<l} \mathcal{S}_1^{l'}$ and $\bfa{p}' \in \bigcup_{l'<l} \bfa{\mathcal{S}}_2^{l'}$.
Using both the correction terms $\bar{N^i_j}^{c}(\cdot,p_j)\,, \bar{M_j^c}(\cdot, \bfa{p})$ in $\mathcal{V}_{L-l+1}$ (the coarser FE spaces with lower accuracy)
and the macro-grid interpolation terms $\bar{N^i_k}(\cdot,p_j')\,, \bar{M_j}(\cdot,\bfa{p}')$ in $\mathcal{V}_{L-l'+1}$ (the finer FE spaces with higher accuracy),
we let
\begin{equation}
\label{eq:main25''}
\bar{N^i_j}(\cdot,p_j)=\bar{N^i_j}^{c}(\cdot, p_j)+\bar{N^i_j}(\cdot,p_j')\,, 
\quad
\bar{M_j}(\cdot,\bfa{p})=\bar{M_j^c}(\cdot, \bfa{p})+\bar{M_j}(\cdot,\bfa{p}')
\end{equation}
(in $\mathcal{V}_{L-l+1}$) be FE approximations for $N^i_j(\cdot,p_j)$, $M_j(\cdot,\bfa{p})$ (in $\mathcal{V}$) respectively.

Note that one can easily verify that $\bar{N^i_j}(\cdot,p_j), \
\bar{M_j}(\cdot,\bfa{p})$ in \eqref{eq:main25''} satisfy \eqref{eq:main11db} as ${\bar{N^i_j}}^c (\cdot,p_j), \
{\bar{M_j^c}} (\cdot,\bfa{p})$ are from \eqref{eq:main52'dbc} and $\bar{N^i_j}(\cdot,p_j'), \
\bar{M_j}(\cdot,\bfa{p}')$ are from \eqref{eq:main11db}.  See \ref{contproof}, for a direction of proof.

\begin{remark}
 One can interchange Steps 1 and 2 in order.  The reason lies in the relationship between the refining of the hierarchical macrogrids (of points) and the error coarsening factor $\varkappa$ of the nested FE spaces.  That is, for a sparser macrogrid of points in the hierarchy, we use some finer resolution FE solve.  Vice versa, the nearer the macro-grid points are, the coarser the FE space gets.   
\end{remark}

\begin{remark}
 It is as expected that when the hierarchical level $l$ is higher, the corresponding FE space becomes coarser.  That is, for local problems, if one wanted to increase the coarseness of the FE space (to get smaller degrees of freedom  towards reduced computational cost), the number of hierarchical macro-grid points would increase, while the FE error still keeps optimal.  This is the spirit of our strategy. 
\end{remark}

In Appendix \ref{proofthm}, we will prove (Theorem \ref{maintheorem} below) that the approximation \eqref{eq:main25''} in $\mathcal{V}_{L-l+1}$ for each continuous exact solution $N^i_j(\cdot,p_j)$, $M_j(\cdot,\bfa{p})$ from \eqref{eq:main11} (in $\mathcal{V}$) is of the same order of accuracy as when we solve \eqref{eq:main11db} normally in the finest FE space $\mathcal{V}_L$ (see \eqref{eq:FEapprox1} for $i,j=1,2$):
\begin{align}
\label{eq:FEapprox2}
\begin{split}
||N_j^i(\bfa{y},p_j) - \bar{N}^i_j(\bfa{y},p_j)||_{\mathcal{V}}
& \le C2^{-L}\|N_j^i(\bfa{y},p_j)\|_{H^2(Y)}\leq C2^{-L}\|\Delta N_j^i(\bfa{y},p_j)\|_{L^2(Y)}\\
&\leq C2^{-L}\,,\\
||M_j(\bfa{y},\bfa{p}) - \bar{M}_j(\bfa{y},\bfa{p})||_{\mathcal{V}} &\leq C2^{-L}\,.
\end{split}
\end{align}

\bigskip

As in \cite{brown13}, the cell problems \eqref{eq:cell} give us the cell operator.  With $i,j=1,2$, we assume that $\bar{N^{i}_j}(\bfa{y},p_j), \bar{M_j}(\bfa{y},\bfa{p}) \in \mathcal{V}_{L-l+1}$ have been constructed as above (see \eqref{eq:main11db} and \eqref{eq:main25''}) for $p_j \in \mathcal{S}_1^l$ and $\bfa{p}=(p_1,p_2) \in \bfa{\mathcal{S}}_2^l$, respectively.  Then, we have the following convergence results for the hierarchical solve, with level $l = \overline{1,L}\,.$ 
\begin{theorem}
\label{maintheorem}
For a sufficient large constant $C_{1*},\ C_{2*}$, which depend only on the coefficients $k_1,k_2$ (so, on the cell operator of \eqref{eq:cell}), 
we have the error estimates
\begin{equation}
\label{eq:main55'**} 
\begin{split}
&||\nabla_yN^i_j(\bfa{y},p_j)-\nabla_y\bar{N^{i}_j}(\bfa{y},p_j)||_{\bfa{L}^2(Y)} 
\leq C_{1*} l 2^{-L}\,,\\
&||\nabla_y M_j(\bfa{y},\bfa{p})-\nabla_y\bar{M_j}(\bfa{y},\bfa{p})||_{\bfa{L}^2(Y)} \leq C_{2*} l 2^{-L}\,.
\end{split}
\end{equation}
\end{theorem}

\begin{proof}
The proof can be found in Appendix \ref{proofthm}.
\end{proof}

For the degrees of freedom needed to solve the cell problems \eqref{eq:cell}, we have the following theorems.

\begin{theorem}
\label{theoremdegree1}
The total degrees of freedom required to solve \eqref{eq:nj} 
with all points in $\mathcal{S}_1^1, \mathcal{S}_1^2, \dots, \mathcal{S}_1^L $ is $\mathcal{O}(L2^{L+1})$ 
using the hierarchical solve, while it is $\mathcal{O} ((2^{L})^2)$ in the full solve
(where the finest mesh is used for all cell problems at all macro-grid points). 
\end{theorem}
\begin{proof}
Note that the number of macroscopic points in $\mathcal{S}_1^l$ is $\mathcal{O}(2^{l})$. The space $\mathcal{V}_{L+1-l}$ is of dimension $\mathcal{O}(2^{L+1-l})$. 
Thus, the total degrees of freedom for solving \eqref{eq:nj} 
with all points in $\mathcal{S}_1^l$ is $\mathcal{O}(2^{l})\mathcal{O}(2^{L+1-l}) = \mathcal{O}(2^{L+1})$. 
Therefore, the total degrees of freedom required to solve \eqref{eq:nj} 
for all points in $L$ macrogrids $\mathcal{S}_1^1, \mathcal{S}_1^2, \dots, \mathcal{S}_1^L $ is $\mathcal{O}(L 2^{L+1})$, while it is $\mathcal{O} ((2^{L})^2)$ in the full fine mesh solve (where the cell problems are solved with the finest mesh level $\mathcal{V}_L$
(\ref{fesps}) at all macro-grid points in $U_{1,L}=\displaystyle\bigcup_{l=1}^{L} \mathcal{S}_1^l$ (\ref{dU})).
\end{proof}

\begin{theorem}
\label{theoremdegree2}
The total degrees of freedom required to solve \eqref{eq:mj} 
with all points in $\bfa{\mathcal{S}}_2^1, \bfa{\mathcal{S}}_2^2, \dots, \bfa{\mathcal{S}}_2^L $ 
is $\mathcal{O}(L2^{2(L+1)})$
using the hierarchical solve, while it is $\mathcal{O}((2^{2L})^2)$ in the full fine mesh solve. 
\end{theorem}
\begin{proof}
Note that the number of macroscopic points in $\bfa{\mathcal{S}}_2^l$ is $\mathcal{O}(2^{2l})$. The space $\mathcal{V}_{L+1-l}$ is of dimension $\mathcal{O}(2^{2(L+1-l)})$. 
Thus, the total degrees of freedom for solving \eqref{eq:mj} 
with all points in $\bfa{\mathcal{S}}_2^l$ is $\mathcal{O}(2^{2l})\mathcal{O}(2^{2(L+1-l)}) = \mathcal{O}(2^{2(L+1)})$. 
Therefore, the total degrees of freedom required to solve \eqref{eq:mj} 
for all points in $\bfa{\mathcal{S}}_2^1, \bfa{\mathcal{S}}_2^2, \dots, \bfa{\mathcal{S}}_2^L $ is $\mathcal{O}(L 2^{2(L+1)})$, while it is $\mathcal{O} ((2^{2L})^2)$ in the full fine mesh solve (where we solve the cell problems using the finest mesh level $\mathcal{V}_L$ (\ref{fesps}) 
at all macro-grid points in $\bfa{U}_{2,L}=\displaystyle\bigcup_{l=1}^{L} \bfa{\mathcal{S}}_2^l$ \eqref{dU}).
\end{proof}

\section{The proof of Theorem \ref{maintheorem} (convergence results for the hierarchical solve)}
\label{proofthm}

Involving the cell problems \eqref{eq:cell}, we now give a proof of Theorem \ref{maintheorem} that the hierarchical method (developed in Appendix \ref{hier}) reaches the same order of accuracy as the full solve, where the cell problems at every macroscopic point in \eqref{dU} are solved by the finest FE space \eqref{eq:FEapprox2}.  

Here, the notation is as follows: for $i,j=1,2\,,$ 
$p_i,p_i' \in U_{1,L}, 
\bfa{p}=(p_1,p_2), \; \bfa{p}'=(p'_1,p'_2)  \in \bfa{U}_{2,L}$ 
(defined in \eqref{dU}); $\phi_j, \psi_j \in \mathcal{W}=H^1_{\#}(Y)$ (defined in (\ref{trysps})); $l =\overline{1,L}$; the solutions of the cell problems \eqref{eq:cell} are $N^i_j, M_j \in \mathcal{U} = H^2_{\#}(Y) \subset \mathcal{V}=H^1_{\#}(Y) / \mathbb{R}$ (defined in (\ref{fesps})); $\bar{N}^i_j, \bar{M}_j \in \mathcal{V}_{L-l+1}$ (defined in \eqref{eq:main11db} and \eqref{eq:main25''}); 
${\bar{N^i_j}}^c, {\bar{M_j^c}} \in \mathcal{V}_{L-l+1}$ 
(defined in \eqref{eq:main52'dbc}); ${N^i_j}^c, {M_j^c} \in \mathcal{U} \subset \mathcal{V}$ (to be defined in \eqref{dc}); 
${\Bar{\Bar{N^i_j}}}^c, {\Bar{\Bar{M_j^c}}} \in \mathcal{V}_{L-l+1}$ (to be defined in \eqref{eq:main53'}).  
It is required that the coefficients $k_j$ and $Q_j$ from \eqref{eq:cell} satisfy \eqref{Coercivity} and Assumption \ref{Lipschitz}.  Throughout this section, we assume that 
\begin{equation}\label{0hypo}
\dd \int_Y N^i_j \, \dy =  \int_Y  M_j \, \dy = 0\,.
\end{equation}

Before proving the main convergence results for the hierarchical solve (Theorem \ref{maintheorem}), we need several lemmas as follows.

\begin{lemma}
\label{lemma1}
There exists some positive number $C$ such that
\[||\nabla_y N^i_j(\bfa{y},p_j)||_{\bfa{L}^2(Y)}
\leq C\,,\quad
||\nabla_yM_j(\bfa{y},\bfa{p})||_{\bfa{L}^2(Y)}
\leq C\,.\]
\end{lemma}
\begin{proof}
 In Eq.\ (\ref{eq:main11}), substituting $N^i_j$ and $M_j$ into $\phi_j$ and $\psi_j$ respectively, one gets
\begin{equation}
\label{eq:main27'}
\begin{split}
\int_Y k_j(\bfa{y},p_j) \nabla_y N^i_j(\bfa{y},p_j) \cdot \nabla_y N^i_j(\bfa{y},p_j) \, \dy
 &=- \int_Y k_j(\bfa{y},p_j) \bfa{e}^i \cdot \nabla_y N^i_j(\bfa{y},p_j) \, \dy\,,\\
\int_Y k_j(\bfa{y},p_j) \nabla_y M_j(\bfa{y},\bfa{p}) \cdot \nabla_y M_j(\bfa{y},\bfa{p}) \, \dy
&=  \int_Y Q_j(\bfa{y},\bfa{p}) M_j(\bfa{y},\bfa{p})\, \dy\,.
\end{split}
\end{equation}
Then, by \eqref{Coercivity}, Cauchy-Schwarz inequality, and Poincar\'{e} inequality, we have from \eqref{eq:main27'} that
\begin{equation}
\label{eq:main27'''}
\begin{split}
||\nabla_y N^i_j(\bfa{y},p_j)||^2_{\bfa{L}^2(Y)}  \leq C ||\nabla_y N^i_j(\bfa{y},p_j)||_{\bfa{L}^2(Y)} \,,\\
||\nabla_y M_j(\bfa{y},\bfa{p})||^2_{\bfa{L}^2(Y)}  \leq C ||M_j(\bfa{y},\bfa{p})||_{L^2(Y)} \leq C ||\nabla_y M_j(\bfa{y},\bfa{p})||_{\bfa{L}^2(Y)}\,,
\end{split}
\end{equation}
for some $C > 0$. We thus obtain
\begin{equation}
\label{eq:main27'''r1}
\begin{split}
||\nabla_y N^i_j(\bfa{y},p_j)||_{\bfa{L}^2(Y)}  \leq C\,, \quad
||\nabla_y M_j(\bfa{y},\bfa{p})||_{\bfa{L}^2(Y)}  \leq C \,.
\end{split}
\end{equation}
These prove the Lemma.
\end{proof}

For the proofs of the later lemmas and thus for proving the main Theorem \ref{maintheorem}, we define the continuous correction terms 
(which are in fact the terms to be corrected) by
%
\begin{equation}\label{dc}
 {N^i_j}^c(\cdot,p_j)={N^i_j}(\cdot,p_j)-{N^i_j}(\cdot,p_j'), \qquad \ {M_j^c}(\cdot,\bfa{p})={M_j}(\cdot,\bfa{p})-{M_j}(\cdot,\bfa{p}')\,.
\end{equation}
These $ {N^i_j}^c(\cdot,p_j), \ {M_j^c}(\cdot,\bfa{p})$ ($\in \mathcal{U} = H^2_{\#}(Y) \subset \mathcal{V}=H^1_{\#}(Y) / \mathbb{R}$) 
hence satisfy Eq.\ (\ref{eq:main11}), 
and we get from there the following equations:  
\begin{equation}
\label{eq:main52dc} 
\begin{split}
\int_Y &k_j(\bfa{y},p_j) \nabla_y {N^{i}_j}^c(\bfa{y},p_j) \cdot  \nabla_y \phi_j(\bfa{y}) \, \dy\\ = & 
-\int_Y (k_j(\bfa{y},p_j) - k_j(\bfa{y},p_j'))\nabla_y {N^{i}_j}(\bfa{y},p_j') \cdot \nabla_y \phi_j(\bfa{y}) \, \dy\\
& -\int_Y (k_j(\bfa{y},p_j) - k_j(\bfa{y},p_j')) \bfa{e}^i \cdot \nabla_y \phi_j(\bfa{y}) \, \dy\,,
\end{split}
\end{equation}
\begin{equation}
\label{eq:main52'dc} 
\begin{split}
\int_Y &k_j(\bfa{y},p_j) \nabla_y {M_j^c}(\bfa{y},\bfa{p}) \cdot  \nabla_y \psi_j(\bfa{y}) \, \dy\\ = & 
-\int_Y (k_j(\bfa{y},p_j) - k_j(\bfa{y},p'_j))\nabla_y M_j(\bfa{y},\bfa{p}') \cdot \nabla_y \psi_j(\bfa{y}) \, \dy \\
& -\int_Y (Q_j (\bfa{y},\bfa{p})-Q_j (\bfa{y},\bfa{p}'))\psi_j(\bfa{y}) \, \dy\,,
\end{split}
\end{equation}
for $i,j = 1,2$ and for all $\phi_j, \psi_j \in \mathcal{W}$ defined in (\ref{trysps}).  Indeed, it follows from \eqref{eq:main11} that

\begin{align}\label{contproof}
\begin{split}
 \int_Y &k_j(\bfa{y},p_j) \nabla_y {N^i_j}^c(\bfa{y},p_j) \cdot  \nabla_y \phi_j(\bfa{y}) \, \dy\\ 
 = & \int_Y k_j(\bfa{y},p_j) \nabla_y {N^i_j}(\bfa{y},p_j) \cdot  \nabla_y \phi_j(\bfa{y}) \, \dy
 -\int_Y k_j(\bfa{y},p_j) \nabla_y {N^i_j}(\bfa{y},p_j') \cdot  \nabla_y \phi_j(\bfa{y}) \, \dy\\
 =& - \int_Y k_j(\bfa{y},p_j) \bfa{e}^i \cdot \nabla_y \phi_j(\bfa{y})\, \dy  
 -\int_Y k_j(\bfa{y},p_j) \nabla_y {N^i_j}(\bfa{y},p_j') \cdot  \nabla_y \phi_j(\bfa{y}) \, \dy\\
 =& - \int_Y k_j(\bfa{y},p_j) \bfa{e}^i \cdot \nabla_y \phi_j(\bfa{y})\, \dy - \int_Y (k_j(\bfa{y},p_j) - k_j(\bfa{y},p_j')) \nabla_y {N^i_j}(\bfa{y},p_j') \cdot  \nabla_y \phi_j(\bfa{y}) \, \dy \\
  & \hspace{5pt} - \int_Y  k_j(\bfa{y},p_j') \nabla_y {N^i_j}(\bfa{y},p_j') \cdot  \nabla_y \phi_j(\bfa{y}) \, \dy \\
  =  & - \int_Y (k_j(\bfa{y},p_j) - k_j(\bfa{y},p_j')) \nabla_y {N^i_j}(\bfa{y},p_j') \cdot  \nabla_y \phi_j(\bfa{y}) \, \dy - \int_Y k_j(\bfa{y},p_j) \bfa{e}^i \cdot \nabla_y \phi_j(\bfa{y})\, \dy\\
  & \hspace{5pt} + \int_Y k_j(\bfa{y},p_j') \bfa{e}^i \cdot \nabla_y \phi_j(\bfa{y})\, \dy\,.
\end{split}
\end{align}
Therefore, \eqref{eq:main52dc} is justified, and similarly for \eqref{eq:main52'dc}.  
\begin{lemma}
\label{lemma_gradlips}
There exists some positive constant $C$ such that \\
$||\nabla_y {N^i_j}^c(\bfa{y},p_j)||_{\bfa{L}^2(Y)}$ 
$\leq C|p_j-p_j'|\,,\qquad 
||\nabla_y M_j^c(\bfa{y},\bfa{p})||_{\bfa{L}^2(Y)}$
$\leq C |\bfa{p}-\bfa{p}'|$\,. \\ 
\end{lemma}
\begin{proof}
Substituting ${N^i_j}^c\,, {M_j^c}$ (defined in \eqref{dc}) into $\phi_j, \psi_j$ respectively in the system (\ref{eq:main52dc})--(\ref{eq:main52'dc}), one obtains
\begin{align*}
 \begin{split}
\int_Y & k_j(\bfa{y},p_j) \nabla_y {N^{i}_j}^c(\bfa{y},p_j) \cdot  \nabla_y {N^{i}_j}^c(\bfa{y},p_j) \, \dy \nonumber \\ 
= & -\int_Y (k_j(\bfa{y},p_j) - k_j(\bfa{y},p_j'))\nabla_y N^i_j(\bfa{y},p_j') \cdot \nabla_y {N^{i}_j}^c(\bfa{y},p_j) \, \dy \\
&-\int_Y (k_j(\bfa{y},p_j) - k_j(\bfa{y},p_j')) \bfa{e}^i \cdot \nabla_y {N^{i}_j}^c(\bfa{y},p_j) \, \dy\,,  
\end{split}
\end{align*}
\begin{align}\label{eq:main29'}
\begin{split}
\int_Y & k_j(\bfa{y},p_j) \nabla_y {M_j^c}(\bfa{y},\bfa{p}) \cdot  \nabla_y {M_j^c}(\bfa{y},\bfa{p})\, \dy   \\ 
= &-\int_Y (k_j(\bfa{y},p_j) - k_j(\bfa{y},p_j'))\nabla_y M_j(\bfa{y},\bfa{p}') \cdot \nabla_y {M_j^c}(\bfa{y},\bfa{p})\, \dy  \\
& -\int_Y (Q_j(\bfa{y},\bfa{p})-Q_j(\bfa{y},\bfa{p}')){M_j^c}(\bfa{y},\bfa{p})\, \dy \,. 
\end{split}
\end{align}
Note that for each fixed $p_j' \in U_{1,L}\,,$ it holds that $\nabla_y N^i_j(\bfa{y},p_j')$ and $\nabla_y M_j(\bfa{y},\bfa{p}')$ are uniformly bounded in $\bfa{L}^2(Y)$ by Lemma \ref{lemma1}.  From this remark, Assumption \ref{Lipschitz}, inequalities (\ref{Coercivity}) and Cauchy-Schwarz inequality, we have from \eqref{eq:main29'} that
\begin{equation}
\label{eq:main32'} 
\begin{split}
&||\nabla_y {N^i_j}^c||^2_{\bfa{L}^2(Y)}  \leq C|p_j - p_j'| \cdot ||\nabla_y {N^i_j}^c||_{\bfa{L}^2(Y)}\,,\\
&||\nabla_y {M_j^c}||^2_{\bfa{L}^2(Y)}\leq C|\bfa{p} - \bfa{p}'| \cdot ||\nabla_y {M_j^c}||_{\bfa{L}^2(Y)}\,,
\end{split}
\end{equation}
for all $\bfa{y}$ in $Y$. The last inequality comes from the fact that 
$|p_j - p'_j| \leq |\bfa{p} - \bfa{p}'|$, where $\bfa{p} = (p_1,p_2), \bfa{p}' = (p'_1,p'_2)\,.$
\end{proof}
We now use the assumption that our media are isotropic and thus each hydraulic conductivity becomes a function $k_i(\bfa{y},p_i)$ multiplying with the identity matrix, for $i=1,2\,.$ 

\begin{lemma}
\label{lemma2}
There exists some positive constant $C$ such that
\[||\Delta_y N^i_j(\bfa{y},p_j)||_{L^2(Y)}
\leq \ C\,,
\qquad ||\Delta_y M_j(\bfa{y},\bfa{p})||_{L^2(Y)}
\leq \ C\,.\]
\end{lemma}
\begin{proof}
The cell problem (\ref{eq:cell}) can be rewritten as follows:
%
\begin{equation}
\label{eq:main32''} 
\begin{split}
k_j\Delta_y N^i_j + \nabla_y k_j \cdot \nabla_y N^i_j + \div_y (k_j \bfa{e}^i) = 0\,,
\end{split}
\end{equation}
\begin{equation}
\label{eq:main32'''} 
\begin{split}
k_j\Delta_y M_j + \nabla_y k_j \cdot \nabla_y M_j + Q_j = 0\,.
\end{split}
\end{equation}
Rearranging these equations, we have
\begin{equation}
\label{eq:main32*} 
\begin{split}
\Delta_y N^i_j = -\frac{1}{k_j} (\nabla_y k_j \cdot  \nabla_y N^i_j+ \div_y (k_j \bfa{e}^i))\,,
\end{split}
\end{equation}
\begin{equation}
\label{eq:main32**} 
\begin{split}
\Delta_y M_j = -\frac{1}{k_j} ( \nabla_y k_j \cdot \nabla_y M_j + Q_j)\,.
\end{split}
\end{equation}
By (\ref{Coercivity}) and Lemma \ref{lemma1}, there exists some positive constant $C$ such that 
\[||\Delta_y N^i_j(\bfa{y},p_j)||_{L^2(Y)} \leq C, \quad ||\Delta_y M_j(\bfa{y},\bfa{p})||_{L^2(Y)} \leq C\,.\]
\end{proof}



\begin{lemma}
\label{lemma_trilips}
For some positive constant $C\,,$ we obtain
\begin{align}\label{eq:main32***}
 \begin{split}
||\Delta_y {N^i_j}^c(\bfa{y},p_j)||_{L^2(Y)}
\leq C|p_j-p_j'| \,,\qquad
||\Delta_y {M_j^c}(\bfa{y},\bfa{p})||_{L^2(Y)} 
\leq C|\bfa{p}-\bfa{p}'| \,.
\end{split}
\end{align}
\end{lemma}
\begin{proof}
From Definition \eqref{dc} and the system \eqref{eq:main52dc}--(\ref{eq:main52'dc}), we deduce that
\begin{equation}
\label{eq:main33}
\begin{split}
& k_j(\bfa{y},p_j) \Delta_y {N^{i}_j}^c(\bfa{y},p_j) +\nabla_y k_j(\bfa{y},p_j) \cdot \nabla_y {N^{i}_j}^c(\bfa{y},p_j) \\
&=  -\nabla_y(k_j(\bfa{y},p_j) - k_j(\bfa{y},p_j')) \cdot \nabla_y N^i_j(\bfa{y},p_j')\\
 & \quad - (k_j(\bfa{y},p_j) - k_j(\bfa{y},p_j'))\Delta_y N^i_j(\bfa{y},p_j')
-\div_y((k_j(\bfa{y},p_j) - k_j(\bfa{y},p_j')) \bfa{e}^i)\,,
\end{split}
\end{equation}
\begin{equation}
\label{eq:main33'}
\begin{split}
& k_j(\bfa{y},p_j) \Delta_y {M_j^c}(\bfa{y},\bfa{p}) +\nabla_y k_j(\bfa{y},p_j) \cdot \nabla_y {M_j^c}(\bfa{y},\bfa{p}) \\
&=  
 -\nabla_y(k_j(\bfa{y},p_j) - k_j(\bfa{y},p_j')) \cdot \nabla_y M_j(\bfa{y},\bfa{p}')\\
 & \quad - (k_j(\bfa{y},p_j) - k_j(\bfa{y},p_j'))\Delta_y M_j(\bfa{y},\bfa{p}')
- (Q_j(\bfa{y},\bfa{p})-Q_j(\bfa{y},\bfa{p}'))\,.
\end{split}
\end{equation}
Therefore, we obtain
\begin{equation}
\label{eq:main33''}
\begin{split}
&\Delta_y {N^{i}_j}^c(\bfa{y},p_j) \\
&= \frac{1}{k_j}\left [-\nabla_y k_j(\bfa{y},p_j) \cdot \nabla_y {N^{i}_j}^c(\bfa{y},p_j) -\nabla_y(k_j(\bfa{y},p_j) - k_j(\bfa{y},p_j'))\cdot \nabla_y N^i_j(\bfa{y},p_j')\right.\\
&\left. \hspace{35pt} - (k_j(\bfa{y},p_j) - k_j(\bfa{y},p_j'))\Delta_y N^i_j(\bfa{y},p_j')
-\div_y((k_j(\bfa{y},p_j) - k_j(\bfa{y},p_j')) \bfa{e}^i)
\right]\,,
\end{split}
\end{equation}
\begin{equation}
\label{eq:main33'''}
\begin{split}
& \Delta_y M_j^c(\bfa{y},\bfa{p}) \\
&= \frac{1}{k_j}\left[- \nabla_y k_j(\bfa{y},p_j) \cdot \nabla_y M_j^c(\bfa{y},\bfa{p})   -\nabla_y(k_j(\bfa{y},p_j) - k_j(\bfa{y},p_j')) \cdot \nabla_y M_j(\bfa{y},\bfa{p}')\right.\\
 &\left. \hspace{55pt} - (k_j(\bfa{y},p_j) - k_j(\bfa{y},p_j'))\Delta_y M_j(\bfa{y},\bfa{p}')
- (Q_j(\bfa{y},\bfa{p})-Q_j(\bfa{y},\bfa{p}'))\right]\,.
\end{split}
\end{equation}
Taking $||\cdot||_{L^2(Y)}$ on both sides of Eqs.\ (\ref{eq:main33''}) and (\ref{eq:main33'''}), then using Assumption \eqref{lipschitz}, Poincar\'{e} inequality, Assumption \ref{Coercivity}, Lemmas \ref{lemma_gradlips}, \ref{lemma1} and \ref{lemma2}, we derive that there exists some positive constant $C$ such that
\begin{equation}
\label{eq:main33*}
\begin{split}
||\Delta_y {N^i_j}^c(\bfa{y},p_j)||_{L^2(Y)} \leq C|p_j-p'_j|\,, \quad ||\Delta_y M_j^c(\bfa{y},\bfa{p})||_{L^2(Y)} \leq C|\bfa{p}-\bfa{p}'|\,. \\
\end{split}
\end{equation} 

\end{proof}
\begin{lemma}
\label{lemma3}
For some positive constant $C\,,$ we have
\begin{align}\label{lb5ine}
\begin{split}
||{N^i_j}^c(\bfa{y},p_j)||_{L^2(Y)}
\leq C|p_j-p_j'|\,,\qquad
||M_j^c(\bfa{y},\bfa{p})||_{L^2(Y)}
\leq C|\bfa{p}-\bfa{p}'|\,.
\end{split}
\end{align}
\end{lemma}
\begin{proof}

It follows from Definition \eqref{dc} and Assumption \eqref{0hypo} that
\[\int_Y {N^i_j}^c \, \dy = 
\int_Y M_j^c \, \dy = 0\,.\]
By Poincar\'{e}-Wirtinger inequality and Lemma \ref{lemma_gradlips},
there holds Lemma \ref{lemma3}.
\end{proof}
\begin{lemma}
\label{lemma_clips}
There exists a positive constant $C$ such that
\[||{N^{i}_j}^c(\bfa{y},p_j)||_{H^2(Y)} \leq C|p_j-p_j'|\,,\qquad
||M_j^c(\bfa{y},\bfa{p})||_{H^2(Y)}\leq C |\bfa{p}-\bfa{p}'|\,.\]
\end{lemma}
\begin{proof}
%
%

%

Let $\omega \subset \mathbb{R}^2$ be a region such that $Y \subset \omega\,.$ 
Let $\phi \in {C}^\infty_0(\omega)$ 
satisfy $\phi = 1$ in $Y\,.$  Recalling Definition \eqref{dc} for the continuous ${N^i_j}^c$ and their spaces, we have
\begin{equation}
\label{eq:main34}
\begin{split}
\Delta_y(\phi {N^i_j}^c) = (\Delta_y \phi) {N^i_j}^c + 2 (\nabla_y \phi) \cdot (\nabla_y {N^i_j}^c) + \phi \Delta_y {N^i_j}^c\,.
\end{split}
\end{equation}
Since $\phi {N^i_j}^c = 0$ on $\partial \omega$, one can apply the elliptic regularity theorem (the boundary $H^2$-regularity, see \cite{evans}, for instance) to get $\phi {N^i_j}^c \in H^2(\omega)$.  We thus have the estimate 
\begin{equation}
\label{eq:main35}
\begin{split}
||{N^i_j}^c||_{H^2(Y)} \leq ||\phi {N^i_j}^c||_{H^2(\omega)} \leq ||\Delta_y(\phi {N^i_j}^c)||_{L^2(\omega)}\,.
\end{split}
\end{equation}
Therefore,
\begin{equation}
\label{eq:main36}
\begin{split}
&||{N^i_j}^c||_{H^2(Y)} \leq ||(\Delta_y \phi) {N^i_j}^c||_{L^2(\omega)} + || 2 (\nabla_y\phi) \cdot (\nabla_y {N^i_j}^c)||_{L^2(\omega)} + ||\phi \Delta_y {N^i_j}^c||_{L^2(\omega)}\,.\\
\end{split}
\end{equation}
Now, since $\phi$ is smooth and compactly supported, it follows that
\begin{equation}
\label{eq:main36'}
|\Delta_y \phi|, \ \|\nabla_y \phi\|_{\bfa{L}^2(Y)}, \ |\phi| \leq C\,. 
\end{equation}
By Lemmas \ref{lemma_gradlips}, \ref{lemma_trilips} and \ref{lemma3}, we get
\begin{equation}
\label{eq:main37}
\begin{split}
&||{N^i_j}^c||_{L^2(Y)}, \ ||\nabla_y {N^i_j}^c||_{\bfa{L}^2(Y)}, \ ||\Delta_y {N^i_j}^c||_{L^2(Y)} \leq C|p_j-p_j'|\,.
%
\end{split}
\end{equation}
Due to the $Y$-periodicity of ${N^i_j}^c$, ${M_j^c}$ in the domain $\omega$, we thus deduce from \eqref{eq:main37} that
\begin{equation}
\label{eq:main38}
\begin{split}
&||{N^i_j}^c||_{L^2(\omega)}, \ ||\nabla_y {N^i_j}^c||_{\bfa{L}^2(\omega)}, \ ||\Delta_y {N^i_j}^c||_{L^2(\omega)} \leq C_1|p_j-p_j'|\,.
\end{split}
\end{equation}
Then by (\ref{eq:main36'}) and (\ref{eq:main38}), we can easily see from (\ref{eq:main36}) that for some $C > 0\,,$ 
\begin{equation}
\label{eq:main38'}
||{N^i_j}^c||_{H^2(Y)} \leq C|p_j-p_j'|\,.
\end{equation}

Similarly, $||M_j^c||_{H^2(Y)} \leq C|\bfa{p}-\bfa{p}'|\,.$
\end{proof}
For the proofs of the next lemmas and thus for proving the main Theorem \ref{maintheorem}, we consider the following problems:
find $ {\Bar{\Bar{N^i_j}}}^c(\bfa{y},p_j)$, ${\Bar{\Bar{M_j^c}}}(\bfa{y},\bfa{p}) \in \mathcal{V}_{L+1-l}$ such that 
\begin{equation*}
\label{eq:main53} 
\begin{split}
\int_Y &k_j(\bfa{y},p_j) \nabla_y {\Bar{\Bar{N^i_j}}}^c(\bfa{y},p_j) \cdot  \nabla_y \phi_j(\bfa{y})\, \dy\\ = & 
-\int_Y (k_j(\bfa{y},p_j) - k_j(\bfa{y},p_j'))\nabla_y N^i_j(\bfa{y},p_j') \cdot \nabla_y \phi_j(\bfa{y})\, \dy\\
& -\int_Y (k_j(\bfa{y},p_j) - k_j(\bfa{y},p_j')) \bfa{e}^i \cdot \nabla_y \phi_j(\bfa{y})\, \dy\,,
\end{split}
\end{equation*}
\begin{align}
\label{eq:main53'} 
\begin{split}
\int_Y &k_j(\bfa{y},p_j) \nabla_y {\Bar{\Bar{M_j^c}}}(\bfa{y},\bfa{p}) \cdot  \nabla_y \psi_j(\bfa{y})\, \dy\\ = & 
-\int_Y (k_j(\bfa{y},p_j) - k_j(\bfa{y},p'_j))\nabla_y M_j(\bfa{y},\bfa{p}') \cdot \nabla_y \psi_j(\bfa{y})\, \dy\\
& -\int_Y (Q_j(\bfa{y},\bfa{p})-Q_j(\bfa{y},\bfa{p}'))\psi_j(\bfa{y}) \, \dy\,,
\end{split}
\end{align}
for all $\phi_j, \psi_j \in \mathcal{W}_{L+1-l}\,.$
Note that this system \eqref{eq:main53'} is obtained by replacing the continuous $ {N^i_j}^c(\cdot,p_j), \ {M_j^c}(\cdot,\bfa{p}) \in \mathcal{V}$ in \eqref{eq:main52dc}--\eqref{eq:main52'dc} with $ {\Bar{\Bar{N^i_j}}}^c(\bfa{y},p_j), \ {\Bar{\Bar{M_j^c}}}(\bfa{y},\bfa{p}) \in \mathcal{V}_{L+1-l}\,.$ 
\begin{lemma}
\label{theorem5}
There exist some positive constants $C_1,C_2$ such that
\begin{equation}
\label{eq:main44''} 
\begin{split}
||\nabla_y{N^{i}_j}^c(\bfa{y},p_j) - \nabla_y{\Bar{\Bar{N^i_j}}}^c(\bfa{y},p_j)||_{\bfa{L}^2(Y)} &\leq C_1 2^{-L},\\
||\nabla_y{M_j^c}(\bfa{y},\bfa{p}) - \nabla_y{\Bar{\Bar{M_j^c}}}(\bfa{y},\bfa{p})||_{\bfa{L}^2(Y)} &\leq C_2 2^{-L}\,.
\end{split}
\end{equation}
\end{lemma}
\begin{proof}
It follows respectively from C\'{e}a's lemma \cite{reg-fem}, 
Definition (\ref{dc}) together with Assumption (\ref{eq:FEapprox}), and Lemma \ref{lemma_clips} that
\begin{align}
\label{eq:main44} 
\begin{split}
||\nabla_y{N^{i}_j}^c(\bfa{y},p_j) - \nabla_y{\Bar{\Bar{N^i_j}}}^c(\bfa{y},p_j)||_{\bfa{L}^2(Y)} 
& \leq C \inf_{\phi\in{\mathcal V}_{L-l+1}}\|\nabla_y({N^{i}_j}^c-\phi)\|_{\bfa{L}^2(Y)} \\
&\leq C 2^{-(L-l+1)} 
||{N^{i}_j}^c||_{H^2(Y)} \\
& \leq C 2^{-(L-l+1)} |p_j-p_j'| 
\leq C_1 2^{-L}\,.
\end{split}
\end{align}
%
The last inequalities of \eqref{eq:main44} follow from \eqref{dS}, \eqref{dU} as well as 
\begin{equation}\label{sdis}
|p_j-p_j'|, |\bfa{p} - \bfa{p}'| \leq \sqrt{2} \cdot 2^{-l+1}
\end{equation}
(see \cite{brown13}), where the factor $\sqrt{2}$ is absorbed into $C_1\,.$

Similarly, $||\nabla_y{M_j^c}(\bfa{y},\bfa{p}) - \nabla_y{\Bar{\Bar{M_j^c}}}(\bfa{y},\bfa{p})||_{\bfa{L}^2(Y)} \leq C_2 2^{-L}\,.$
\end{proof}

Now, one subtracts (\ref{eq:main53'}) from (\ref{eq:main52'dbc}), and let $\phi_j =  \bar{N^{i}_j}^c(\bfa{y},p_j)-{\Bar{\Bar{N^{i}_j}}}^c(\bfa{y},p_j)$,\  $\psi_j =  \bar{M}_j^c(\bfa{y},\bfa{p})-{\Bar{\Bar{M_j^c}}}(\bfa{y},\bfa{p})$ to obtain
\begin{equation*}
\label{eq:main54} 
\begin{split}
\int_Y &k_j(\bfa{y},p_j) \nabla_y (\bar{N^{i}_j}^c(\bfa{y},p_j)-{\Bar{\Bar{N^{i}_j}}}^c(\bfa{y},p_j)) \cdot  \nabla_y (\bar{N^{i}_j}^c(\bfa{y},p_j)-{\Bar{\Bar{N^{i}_j}}}^c(\bfa{y},p_j))\, \dy
\\ = & 
-\int_Y (k_j(\bfa{y},p_j) - k_j(\bfa{y},p_j'))\nabla_y (\bar{N^{i}_j}(\bfa{y},p_j')-N^i_j(\bfa{y},p_j')) \cdot \nabla_y (\bar{N^{i}_j}^c(\bfa{y},p_j)-{\Bar{\Bar{N^{i}_j}}}^c(\bfa{y},p_j))\, \dy\,,
\end{split}
\end{equation*}

\begin{equation}
\label{eq:main54'} 
\begin{split}
\int_Y &k_j(\bfa{y},p_j) \nabla_y ({\bar{M_j^c}}(\bfa{y},\bfa{p})-{\Bar{\Bar{M_j^c}}}(\bfa{y},\bfa{p})) \cdot  \nabla_y (\bar{M_j^c}(\bfa{y},\bfa{p})-{\Bar{\Bar{M_j^c}}}(\bfa{y},\bfa{p}))\, \dy\\
 = & 
-\int_Y (k_j(\bfa{y},p_j) - k_j(\bfa{y},p'_j))\nabla_y (\bar{M_j}(\bfa{y},\bfa{p}')-M_j(\bfa{y},\bfa{p}')) \cdot \nabla_y (\bar{M_j^c}(\bfa{y},\bfa{p})-{\Bar{\Bar{M_j^c}}}(\bfa{y},\bfa{p}))\, \dy\,.
\end{split}
\end{equation}

Recall that $i,j=1,2\,, l=\overline{1,L}\,;$ $N^i_j, M_j \in \mathcal{U} = H^2_{\#}(Y) \subset \mathcal{V}=H^1_{\#}(Y) / \mathbb{R}$ (defined in (\ref{fesps})); $\bar{N}^i_j, \bar{M}_j \in \mathcal{V}_{L-l+1}$ (defined in \eqref{eq:main11db} and \eqref{eq:main25''}). We are ready to prove the following lemma.  
\begin{lemma}
\label{theorem4}
There exist some positive constants $C_{1,l}\,, \, C_{2,l}$ which only depend on the coefficients $k_1,k_2$ (so, the cell operator of \eqref{eq:cell}) as well as on the respectively level $\mathcal{S}_1^l$ of $p_j$ in $U_{1,L}$ and $\bfa{\mathcal{S}}_2^l$ of $\bfa{p}=(p_1,p_2)$ in $\bfa{U}_{2,L}$ such that for $i,j=1,2\,, l = \overline{1,L}\,,$ we have
\begin{equation}
\label{eq:main51'} 
\begin{split}
&||\nabla_y\bar{N^{i}_j}(\bfa{y}, p_j)-\nabla_yN^i_j(\bfa{y}, p_j)||_{\bfa{L}^2(Y)} \leq C_{1,l} 2^{-L},\\
&||\nabla_y\bar{M_j}(\bfa{y},\bfa{p})-\nabla_yM_j(\bfa{y},\bfa{p})||_{\bfa{L}^2(Y)} \leq C_{2,l} 2^{-L}\,.\\
\end{split}
\end{equation}
\end{lemma}
\begin{proof}
We prove this Lemma by induction. The conclusion \eqref{eq:main51'} obviously holds for $l = 1\,.$  
With $l\geq 2$, assuming that for all $p_j' \in \mathcal{S}_1^{l'}$ and $\bfa{p_j'}=(p_1',p_2') \in \bfa{\mathcal{S}}_2^{l'}$ where $l' \leq l-1$ (so 
$\mathcal{V}_{L-l+1} \subset \mathcal{V}_{L-l'+1}$, in which the later FE space is finer as well as has higher accuracy, and noting that $\bar{N^{i}_j}(\bfa{y},p_j') \in \mathcal{V}_{L-l'+1}$), the conclusion \eqref{eq:main51'} holds for $l' \leq l-1$, that is,
\begin{equation}
\label{eq:main55'} 
\begin{split}
||\nabla_y\bar{N^{i}_j}(\bfa{y},p_j')-\nabla_yN^i_j(\bfa{y},p_j')||_{\bfa{L}^2(Y)} \leq C_{1,l'}\,  2^{-L} \leq C_{1,l-1} \, 2^{-L}\,.
\end{split}
\end{equation}
We will prove that \eqref{eq:main51'} holds for $l\,.$
Using this system of inequalities (\ref{eq:main55'}), Assumptions (\ref{Lipschitz}) and  (\ref{Coercivity}), and $|p_j-p_j'|, |\bfa{p} - \bfa{p}'| \leq \sqrt{2} \cdot 2^{-l+1}$ in \eqref{sdis}, it is not difficult to show from the system (\ref{eq:main54'}) that
\begin{equation}
\label{eq:main56} 
\begin{split}
||\nabla_y ({\bar{N^{i}_j}}^c(\bfa{y},p_j)-{\Bar{\Bar{N^i_j}}}^c(\bfa{y},p_j))||_{\bfa{L}^2(Y)} \leq \gamma C_{1,l-1} 2^{(-L-l+1)}\,,\\
%
\end{split}
\end{equation}
where $\gamma > 0$ is independent of $p_j\,, \bfa{p}\,, l\,,$ and the nested FE spaces, 
while the factor $\sqrt{2}$ in \eqref{sdis} is absorbed into $\gamma$. 
By Lemma \ref{theorem5} and (\ref{eq:main56}), we get 
\begin{equation}\label{eq:main57}
\begin{split}
&|| \nabla_y{N^{i}_j}^c(\bfa{y},p_j) - \nabla_y{\Bar{N^i_j}}^c(\bfa{y},p_j)||_{\bfa{L}^2(Y)} \\ &\leq ||\nabla_y{N^{i}_j}^c(\bfa{y},p_j) - \nabla_y{\Bar{\Bar{N^i_j}}}^c(\bfa{y},p_j)||_{\bfa{L}^2(Y)}
 + ||\nabla_y{\Bar{{N^{i}_j}}}^c(\bfa{y},p_j) - \nabla_y{\Bar{\Bar{N^i_j}}}^c(\bfa{y},p_j)||_{\bfa{L}^2(Y)} \\
 & \leq C_{1} 2^{-L} + \gamma C_{1,l-1} 2^{(-L-l+1)}\,,
 \end{split}
 \end{equation}
 %
where $C_1,\ C_2$ are from  (\ref{eq:main44''}) of Lemma \ref{theorem5}.
Recall that by definitions (\ref{dc}) and (\ref{eq:main25''}) respectively, we have in $\mathcal{V}$ (for $j = 1,2$) that 
\beq\label{dc-db}
\bsp
 {N^i_j}^c(\cdot,p_j)={N^i_j}(\cdot,p)-{N^i_j}(\cdot,p_j'), \ 
 {\Bar{N^i_j}}^c(\cdot,p_j)={\Bar{N}^i_j}(\cdot,p_j)-{\Bar{N}^i_j}(\cdot,p_j')\,.
 \end{split} 
 \eeq
From these expressions \eqref{dc-db}, (\ref{eq:main55'}) and (\ref{eq:main57}), we obtain
\begin{align}
 \label{eq:main56'}
\begin{split}
&||\nabla_y\bar{N^{i}_j}(\bfa{y},p_j)-\nabla_yN^i_j(\bfa{y},p_j)||_{\bfa{L}^2(Y)} \\
& \leq \norm{\nabla_y\left((N^i_j(\bfa{y},p_j)- N^i_j(\bfa{y},p_j')) - (\bar{N^{i}_j}(\bfa{y},p_j)-\bar{N^{i}_j}(\bfa{y},p_j'))\right)}_{\bfa{L}^2(Y)} \\
& \hspace{10pt} + ||\nabla_y(N^i_j(\bfa{y},p_j')-\bar{N^{i}_j}(\bfa{y},p_j'))||_{\bfa{L}^2(Y)} \\
 & = || \nabla_y{N^{i}_j}^c(\bfa{y},p_j) - \nabla_y{\Bar{N^i_j}}^c(\bfa{y},p_j)||_{\bfa{L}^2(Y)} + 
 ||\nabla_y\bar{N^{i}_j}(\bfa{y},p_j')-\nabla_yN^i_j(\bfa{y},p_j')||_{\bfa{L}^2(Y)}\\
&\leq C_{1,l}2^{-L}\,,
\end{split}
\end{align}
%
at the $l$th level, where 
\begin{equation}
\label{eq:main55'*} 
\begin{split}
C_{1,l} = (C_{1} + \gamma C_{1,l-1} 2^{-l+1} + C_{1,l-1} )\,.
%
\end{split}
\end{equation}

Similarly, $||\nabla_y\bar{M_j}(\bfa{y},\bfa{p})-\nabla_yM_j(\bfa{y},\bfa{p})||_{\bfa{L}^2(Y)} \leq C_{2,l} 2^{-L}\,.$
\end{proof}

\bigskip

Let the assumptions of Lemma \ref{theorem4} hold, we now prove Theorem \ref{maintheorem} (convergence results for the hierarchical solve).  For convenience, we state Theorem \ref{maintheorem} again.
\begin{theorem}[Theorem \ref{maintheorem}]
For some positive constants $C_{1*}\,, \, C_{2*}$, which depend only on the coefficients $k_1,k_2$ (so, the cell operator of \eqref{eq:cell}) respectively, we have the following error estimates:
\begin{subequations}
\label{eq:main55'**B} 
\begin{align}
\label{mainthm1}
&||\nabla_yN^i_j(\bfa{y},p_j)-\nabla_y\bar{N^{i}_j}(\bfa{y},p_j)||_{\bfa{L}^2(Y)}
\leq C_{1*} l 2^{-L}\,,\\
\label{mainthm2}
&||\nabla_yM_j(\bfa{y},\bfa{p})-\nabla_y\bar{M_j}(\bfa{y},\bfa{p})||_{\bfa{L}^2(Y)}
\leq C_{2*} l 2^{-L}\,.
\end{align}
\end{subequations}
\end{theorem}

\begin{proof}
We consider a constant $\bar{l}$ 
independent of $L$ and $l$ such that for $l-1 > \bar{l}$,
\begin{equation}\label{barl}
(l-1)2^{-l+1} < \dd \frac{1}{2\gamma}\,.
\end{equation} Involving $l-1 \leq \bar{l}$, we let
\begin{equation}
\label{eq:main55*} 
\begin{split}
C_{1*} = \displaystyle\max \bigg\{\max_{0 \leq l-1 \leq \bar{l}} \Big\{\frac{C_{1,l-1}}{l-1}\Big\}, 2 C_1\bigg\},
\end{split}
\end{equation}
where $C_{1,l-1}$ and the constant $C_1$ (independent of $l$) are from (\ref{eq:main55'*}) and \eqref{eq:main44''} respectively, as in the proof of Lemma \ref{theorem4}. Now, we show by induction that
\begin{equation}
\label{eq:main55**} 
\begin{split}
||\nabla_yN^i_j(\bfa{y},p_j)-\nabla_y\bar{N^{i}_j}(\bfa{y},p_j)||_{\bfa{L}^2(Y)} \leq C_{1*} l 2^{-L}\,.
\end{split}
\end{equation}
 
First, this inequality \eqref{eq:main55**} is easily seen to hold for all $l-1 \leq \bar{l}$.  Indeed, by \eqref{eq:main51'} of Lemma \ref{theorem4} (for the level $\mathcal{S}_1^{l-1}$ of $p_j'$ in $U_{1,L}$) and by (\ref{eq:main55*}) respectively, we get
\begin{equation}
 ||\nabla_y\bar{N^{i}_j}(\bfa{y},p_j')-\nabla_yN^i_j(\bfa{y},p_j')||_{\bfa{L}^2(Y)} \leq C_{1,l-1} 2^{-L} \leq C_{1*}(l-1) 2^{-L}\,.
\end{equation} 

Second, suppose that 
\eqref{eq:main55**} holds for all $(l-1)$th induction steps (including the cases $ l-1 > \bar{l}$). Then, using \eqref{barl} for the induction step $l$th, we obtain from  (\ref{eq:main55'*}) that
\begin{align}
\label{eq:main55***} 
\begin{split}
C_{1,l} &= (C_{1} + \gamma C_{1,l-1} 2^{-l+1} + C_{1,l-1} ) \leq \frac{C_{1*}}{2} + \frac{2^{l-1}}{2(l-1)}\cdot C_{1,l-1} \cdot 2^{-l+1} + (l-1) C_{1*}
\leq C_{1*} l\,,
\end{split}
\end{align}
and the expected inequality \eqref{mainthm1} follows from \eqref{eq:main56'}.  Similarly, \eqref{mainthm2} holds.
\end{proof}

\section{Global convergence of Picard linearization procedure}\label{cp}
%

In this appendix, based on \cite{rpicardc}, we will prove the global convergence of Picard linearization (presented in Section \ref{pre}) for the system \eqref{r1elh} based on \eqref{r1ed}, with $i = 1,2$. 

In \eqref{r1}, 
we assume that the conductivity coefficients $\kappa_i$ satisfies $ 0< \kappa_i \leq d_i\,,$ for some $d_i >0\,.$  Here, we only consider constant conductivities $\kappa_i > 0$ (see \cite{LPicard} for nonconstant conductivities).  For each $i,j=1,2\,,$ the vector-valued function $\bfa{b}_{ij}$ and the function $c_i$ are nonlinear but globally Lipschitz continuous with respectively Lipschitz constants $L_{b_{ij}}, L_{c_i}$ (without any explicit dependence on $\bfa{x}$ and $t$).  Moreover, we assume that each $p_i$ is positive and that $\bfa{b}_{ij}, c_i$ are respectively bounded above by very large constants $\beta_{ij}, C_i\,.$

For simplicity, the subscript $h$ and ($s+1$) are omitted from the Picard iteration (\ref{r1elh}).  We denote by $\| \cdot \|$ the $L^2(\Omega)$-norm and by $(\cdot , \cdot)$ the inner product in $L^2(\Omega)$.  Then, subtracting \eqref{r1ed}
from \eqref{r1elh} and choosing proper $\phi_i = p_i^{n+1} - p_{i,s+1}\,,$
we get in $\bfa{V}_h\,:$
\begin{align}\label{in1}
 \begin{split}
  \frac{1}{\tau} \|p_i^{n+1} - p_{i,s+1}\|^2
   \leq &-b_i(\bfa{p}^{n+1},p_i^{n+1} - p_{i,s+1};\bfa{p}^n) - 
  q_i(\bfa{p}^{n+1},p_i^{n+1} - p_{i,s+1};\bfa{p}^n)\\
  &+b_i(\bfa{p}_{s+1},p_i^{n+1} - p_{i,s+1};\bfa{p}_{s+1}) + 
  q_i(\bfa{p}_{s+1},p_i^{n+1} - p_{i,s+1};\bfa{p}_{s+1})\,.
  \end{split}
  \end{align}
  %

Using \eqref{ai}--\eqref{qi}, the right-hand side of \eqref{in1} is denoted by $RS$ and is written as

\vspace{-15pt}

\begin{align}\label{rhsin1}
 \begin{split}
  RS& = \sum_j \left\{((-\bfa{b}_{ij}(\bfa{p}^n) + \bfa{b}_{ij}(\bfa{p}_{s+1})) \cdot (\nabla p_j^{n+1} - \nabla p_{j,s+1})\,,p_i^{n+1}- p_{i,s+1}) \right.\\
  &\left. \hspace{40pt}  + 2((-\bfa{b}_{ij}(\bfa{p}^n) +\bfa{b}_{ij}(\bfa{p}_{s+1})) \cdot \nabla p_{j,s+1}\,,p_i^{n+1}- p_{i,s+1})\right.\\
  &\left. \hspace{40pt}  +(-\bfa{b}_{ij}(\bfa{p}_{s+1}) \cdot \nabla p_j^{n+1} + \bfa{b}_{ij}(\bfa{p}^n) \cdot \nabla p_{j,s+1}\,,p_i^{n+1}- p_{i,s+1})\right\}\\
  & +\sum_j \left\{ ((-c_i(\bfa{p}^n) +c_i(\bfa{p}_{s+1})) \cdot ((p_i^{n+1} - p_j^{n+1})-(p_{i,s+1} - p_{j,s+1})), p_i^{n+1} - p_{i,s+1})\right.\\
  &\left. \hspace{40pt}  + 2((-c_i(\bfa{p}^n) + c_i(\bfa{p}_{s+1}))\cdot (p_{i,s+1} - p_{j,s+1}), p_i^{n+1} - p_{i,s+1})\right.\\
  &\left. \hspace{40pt} +(-c_i(\bfa{p}_{s+1})(p_i^{n+1} - p_j^{n+1}) + c_i(\bfa{p}^n)(p_{i,s+1} - p_{j,s+1}), p_i^{n+1} - p_{i,s+1})\right\}\,.
 \end{split}
\end{align}

Note that $\| \bfa{p}^{n} - \bfa{p}_{s+1}\| =
\displaystyle \sum_j \|p_j^{n} - p_{j,s+1}\|\,.$  
Assume that there exist sufficient small constant $M_j > 0$ and large constant $D_j > 0$ such that $M_j \leq \|p_j^{n+1} - p_{j,s+1}\| \leq \hat{C}  \|\nabla (p_j^{n+1} - p_{j,s+1})\| \leq D_j\,,$ for $j=1,2\,.$ 
We thus obtain from \eqref{in1} and \eqref{rhsin1} that
\begin{align}\label{in1xi}
 \begin{split}
  &\frac{1}{\tau} \|p_i^{n+1} - p_{i,s+1}\|^2 \\
  & \leq \sum_j \left \{ \|\bfa{b}_{ij}(\bfa{p}^n) - \bfa{b}_{ij}(\bfa{p}_{s+1})\| ( \|\nabla p_j^{n+1} - \nabla p_{j,s+1}\| + 2 \|\nabla p_{j,s+1}\|_{\infty}) \,  \|p_i^{n+1} - p_{i,s+1}\| \right. \\ 
  & \left. \hspace{40pt} + \beta_{ij} \|\nabla p_j^{n+1} - \nabla p_{j,s+1}\| \, \|p_i^{n+1} - p_{i,s+1}\| \right\}\\
  & \quad + \sum_j \left \{ \|(c_i(\bfa{p}^n) - c_i(\bfa{p}_{s+1})\| \, \|(p_i^{n+1} - p_j^{n+1}) - (p_{i,s+1}-p_{j,s+1})\| \, \|p_i^{n+1} - p_{i,s+1}\|\right.\\
  & \left. \hspace{45pt} + \|(c_i(\bfa{p}^n) - c_i(\bfa{p}_{s+1})\| \, (2 \|p_{i,s+1} - p_{j,s+1}\|_{\infty}) \, \|p_i^{n+1} - p_{i,s+1}\| \right.\\
  & \left. \hspace{45pt} + C_i\|(p_i^{n+1} - p_j^{n+1}) - (p_{i,s+1}-p_{j,s+1})\| \, \|p_i^{n+1} - p_{i,s+1}\|\right \}
  \end{split}
  \end{align}
  \begin{align*}
   \begin{split}
  & \leq \sum_j \left \{L_{b_{ij}}\,\| \bfa{p}^n - \bfa{p}_{s+1}\| \, (D_j + 2\|\nabla p_{j,s+1}\|_{\infty}) \|p_i^{n+1} - p_{i,s+1}\|
  + \beta_{ij} \, D_j \|p_i^{n+1} - p_{i,s+1}\|\right \}\\
  & \hspace{7pt} +  \sum_j \left\{L_{c_i}\, \| \bfa{p}^n - \bfa{p}_{s+1}\|  \, (D_j + 2\|\nabla p_{j,s+1}\|_{\infty} ) \|p_i^{n+1} - p_{i,s+1}\| + C_i D_j \|p_i^{n+1} - p_{i,s+1}\|\right\}\,,
  \end{split}
  \end{align*}
  where the constants $\beta_{ij}, C_i$ are assumed to be very large so that the second inequality of \eqref{in1xi} holds.
  
  Upon reorganizing the last inequality of \eqref{in1xi}, we get
  \begin{align}\label{re1}
   \begin{split}
    \|p_i^{n+1} - p_{i,s+1}\| 
    &\leq \sum_j \left\{ L_{b_{ij}} \frac{\tau D_i }{M_i-\tau (\beta_{i1} D_1 + \beta_{i2} D_2)} \, (D_j + 2\|\nabla p_{j,s+1}\|_{\infty})  \, \| \bfa{p}^n - \bfa{p}_{s+1}\|\right \} \\
    &+ \sum_j \left \{L_{c_i} \frac{\tau D_i }{M_i - \tau  C_i (D_1 + D_2)} (D_j + 2\|\nabla p_{j,s+1}\|_{\infty} )  \| \bfa{p}^n - \bfa{p}_{s+1}\|\right \}\,.
   \end{split}
  \end{align}

  
%
Since $\| \bfa{p}^{n+1} - \bfa{p}_{s+1}\| =
\|p_1^{n+1} - p_{1,s+1}\| + \|p_2^{n+1} - p_{2,s+1}\|$, it follows that
\begin{equation}\label{re3}
 \| \bfa{p}^{n+1} - \bfa{p}_{s+1}\| \leq 
 L \frac{\tau D}{M - 2\tau  \beta D} (8D + 8 \|\nabla \bfa{p}_{s+1} \|_{\infty}) \| \bfa{p}^n - \bfa{p}_{s+1}\|\,,
\end{equation}
where $M= \min\{M_j\}\,, D = \max\{D_j\}\,, \beta = \max \{ \beta_{ij}, C_i\}\,,$ $L= \max \{ L_{b_{ij}}, L_{c_i}\}\,.$

Let $\bfa{U}$ be the exact solution to the original problem (\ref{r1}).  Then, using the error estimate in \cite{GalerkinFEM} (Theorem 1.5), we obtain
\begin{equation}\label{re4}
 \| \nabla \bfa{p}_{s+1} \|_{\infty} \leq \|\nabla(\bfa{U}(t_{s+1}) - \bfa{p}_{s+1}) \|_{\infty} + \|\nabla \bfa{U}(t_{s+1})\|_{\infty} \leq C'(\bfa{U})(h + \tau) + \|\nabla \bfa{U}(t_{s+1})\|_{\infty}\,,
\end{equation}
for some constant $C'(\bfa{U})\,,$ where $h$ is the fine scale defined in \eqref{Vh} and $\tau$ is the time step from \eqref{r1ed}.  Redefining constants properly in \eqref{re3}, we finally reach
\begin{equation}\label{re5}
 \| \bfa{p}^{n+1} - \bfa{p}_{s+1}\| \leq 
 C L \frac{\tau D M^{-1}}{1 - 2\tau  \beta D M^{-1}} (1 + (C')(\tau + h)) \| \bfa{p}^{n} - \bfa{p}_{s+1}\|: = \lambda \| \bfa{p}^n - \bfa{p}_{s+1}\|\,.
\end{equation}
The coefficient $\lambda$ will be less than $1$ for sufficient small $\tau$ and $h$, thus it holds that the procedure converges.  In particular, $\lambda \to 0$ as $\tau \to 0$ and $h \to 0$ simultaneously.



\bigskip



\bibliographystyle{plain}
\bibliography{referencesRichards}

\end{document}